\documentclass[11pt]{amsart}
\usepackage[utf8]{inputenc}
\usepackage[total={6.5in,9in},top=1in, left=1in, right=1in, bottom=1in]{geometry}
\usepackage{fancyhdr}
\usepackage{color,amsmath,amssymb,mathtools,graphicx}
\usepackage{amsfonts}
\usepackage{libertine}
\usepackage[libertine]{newtxmath}
\usepackage{bbold}
\usepackage{epsfig,fancyhdr}
\usepackage{dsfont} 
\usepackage{graphicx,float}
\usepackage{graphicx}
\usepackage{epsfig,color,fancyhdr,setspace}
\usepackage{epstopdf}
\usepackage{wrapfig}
\usepackage{soul}
\usepackage{import}
\usepackage{lineno}
\usepackage{stackrel}
\usepackage[dvipsnames]{xcolor}
\usepackage[allcolors = blue,colorlinks]{hyperref}
\usepackage[abs]{overpic}
\usepackage[center]{caption}
\usepackage{subcaption}
\usepackage{pst-node}
\usepackage{tikz-cd}
\usepackage{enumitem}
\counterwithin{figure}{section}
\usepackage{tikz-cd}
\usepackage{wasysym}
\usepackage{marvosym}
\usepackage{extarrows}
\usepackage{amsmath}
\usepackage{graphicx,amsmath,mathtools,float}
\usepackage{setspace}
\graphicspath{{./images/}}
\usepackage{stackrel}
\usepackage{yfonts}
\usepackage{enumitem, hyperref}\makeatletter
\def\namedlabel#1#2{\begingroup
    #2%
    \def\@currentlabel{#2}%
    \phantomsection\label{#1}\endgroup
}
\makeatother
\usepackage{type1cm} 
\usepackage{blindtext}
\usepackage{scalerel,stackengine}
\usepackage{tikz}
\usepackage{mathtools}
\usepackage{amssymb}
\usepackage{tikz}
\usetikzlibrary{shapes,snakes}

\newcommand\quotient[2]{
	\mathchoice
	{% \displaystyle
		\text{\raise1ex\hbox{$#1$}\Big/\lower1ex\hbox{$#2$}}%
	}
	{% \textstyle
		#1\,/\,#2
	}
	{% \scriptstyle
		#1\,/\,#2
	}
	{% \scriptscriptstyle  
		#1\,/\,#2
	}
}

\usepackage{tocbasic}
\newtheorem{theorem}{Theorem}[section]
\newtheorem{lemma}[theorem]{Lemma}
\newtheorem{definition}[theorem]{Definition}
\newtheorem{example}[theorem]{Example}
\newtheorem{proposition}[theorem]{Proposition}
\newtheorem{remark}[theorem]{Remark}

\newtheorem{corollary}[theorem]{Corollary}
\newtheorem{conjecture}[theorem]{Conjecture}

\title{On a new Gram determinant from the M\"obius band}

\author{Dionne Ibarra}
\address{School of Mathematics, Monash University, VIC 3800, Australia.}
\email{{\rm \textcolor{blue}{dionne.ibarra@monash.edu}}}

\author{Gabriel Montoya-Vega}
\address{Department of Mathematics, Queens College and The Graduate Center CUNY, NY, USA}
\email{{\rm \textcolor{blue}{gabriel.montoyavega55@gc.cuny.edu} | \textcolor{blue}{gabrielmontoyavega@gmail.com}}}


\begin{document}
\begin{abstract}
Gram determinants earned traction among knot theorists after E. Witten's presumption about the existence of a 3-manifold invariant connected to the Jones polynomial. Triggered by the creation of such an invariant by N. Reshetikhin and V. Turaev, several mathematicians have explored this line of research ever since. Gram determinants came into play by W. B. Raymond Lickorish's skein theoretic approach to the invariant.
%Being the work of N. Reshetikhin and V. Turaev the trigger of such a connection, several mathematicians have explored this line of research ever since. 
The construction of different bilinear forms is possible through changes in the ambient surface of the Kauffman bracket skein module. Hence, different types of Gram determinants have arisen in knot theory throughout the years; some of these determinants are discussed here. In this article, we introduce a new version of such a determinant from the M\"obius band and prove some important results about its structure. In particular, we explore its connection to the annulus case and factors of its closed formula.
\end{abstract}

\maketitle

\tableofcontents

\section*{Acknowledgments}
The first author acknowledges the support by the Australian Research Council grant DP210103136. The second author acknowledges the support of the National Science Foundation through Grant DMS-2212736.

\section{Introduction}
Gram determinants are named after the Danish mathematician J{\o}rgen Pedersen Gram and they appear in several areas of mathematics including Riemannian geometry, finite element, and machine learning. In knot theory, Gram determinants became of interest following Edward Witten's contemplation of a 3-manifold invariant connected to the Jones polynomial \cite{Wit}. In 1991, a construction of such invariant was presented by Nicolai Reshetikhin and Vladimir Turaev \cite{RT}. Shortly afterwards, W. B. Raymond Lickorish announced a simpler approach to the construction of this invariant in which is considered to be the first modern work on Gram determinants in relation to the mathematical theory of knots \cite{Lic1}. The Gram determinant constructed by Lickorish is known as the Gram determinant of type $A$ and has been extensively studied; see for instance \cite{K-S, DiF, Cai, BIMP}. It is important to remark that in knot theory several matrices arise with a clear connection to Gram determinants. For instance, the Alexander matrix  of a link $L$ whose determinant gives the Alexander polynomial, and the Goeritz matrix which gives the determinant of the link.

\

By changing the ambient surface of the Kauffman bracket skein module, different bilinear forms can be constructed. Rodica Simion investigated bilinear forms of type $B$ while working on chromatic joins \cite{Sim, SchTypeA}. Paul Martin and Hubert Saleur were the first to consider Gram determinants of type $B$ on their work which, among others, enjoys applications to statistical mechanics \cite{MaSa1, MaSa2}. In 2008, J\'ozef H. Przytycki created the notion of Gram determinant of type $Mb$ which results from defining a bilinear form on the M\"obius band \cite{Prz2}. Although a conjecture was presented by Qi Chen during the same year, this line of research was only rigorously pursued about ten years later; see for example \cite{BIMP, BIMP2, PBIMW}. 

\ 

The article is organized as follows. In Section \ref{GramKnot} we present the definition of the relative Kauffman bracket skein module. Moreover, the Gram determinants of type $B$ and type $Mb$ are defined. We introduce the Gram determinant of type $(Mb)_1$, which arises as a modification of the type $Mb$, in Section \ref{GramMB1}. There we present the lollipop method, an innovative tool that helps to prove some important results about the structure of this new Gram determinant. Lastly, an adaptation of Chen's conjecture to this determinant is proposed in Section \ref{FutureDirections}.

\section{Gram determinants in knot theory}\label{GramKnot}
In order to study Gram determinants in knot theory, we need basis elements of a free module over a commutative ring with unity to be described by a manifold with boundary. That is, we need the notion of a relative skein module. Arguably, the most extensively investigated skein module is the Kauffman bracket skein module. In particular, its structure has shown connections between the module and the geometry and topology of the 3-manifold \cite{Prz1}. In this paper, the relative Kauffman bracket skein module plays an important role and is given in Definition \ref{rkbsm}.  
\begin{definition}\label{rkbsm}
Let $M$ be an oriented $3$-manifold and  $\{x_i\}_{1}^{2n}$ be the set of $2n$ framed points on $\partial M$. Let $I=[-1, 1]$, and let $\mathcal{L}^{\mathit{fr}}(2n)$ be the set of all relative framed links (which consists of all framed links in $M$ and all framed arcs, $I \times I$, where $ I \times \partial I$ is connected to framed points on the boundary of $M$) up to ambient isotopy while keeping the boundary fixed in such a way that $L \cap \partial M = \{x_i\}_{1}^{2n}$. Let $R$ be a commutative ring with unity,  $A \in R$ be invertible, and let $S_{2,\infty}^{\mathit{sub}}(2n)$ be the submodule of $R\mathcal{L}^{\mathit{fr}}(2n)$ that is generated by the Kauffman bracket skein relations:

\begin{itemize}
    \item [(i)]$L_+ - AL_0 - A^{-1}L_{\infty}$, and
    \item [(ii)] $L \sqcup \pmb{\bigcirc}  + (A^2 + A^{-2})L$,
\end{itemize}
where \pmb{$\bigcirc$} denotes the framed unknot and the {\it skein triple} $(L_+$, $L_0$, $L_{\infty})$ denotes three framed links in $M$ that are identical except in a small $3$-ball in $M$ where the difference is shown in Figure \ref{KBSM:skeintriple}. Then, the \textbf{relative Kauffman bracket skein module} (RKBSM) of $M$ is the quotient: $$\mathcal{S}_{2,\infty}(M, \{x_i\}_1^{2n}; R, A) = R\mathcal{L}^{\mathit{fr}}(2n) / S_{2,\infty}^{\mathit{sub}}(2n).$$ 
\end{definition}

\begin{figure}[ht]
\centering
\begin{subfigure}{.25\textwidth}
\centering
$\vcenter{\hbox{\includegraphics[scale = .35]{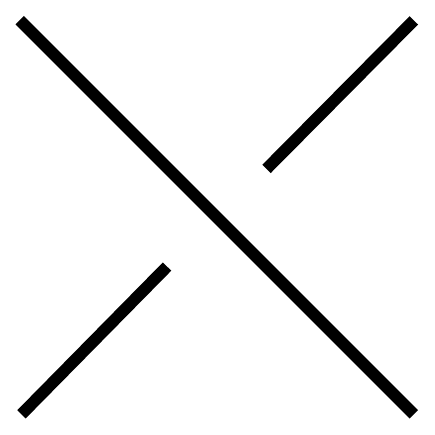}}} $
\caption{$L_+$.} \label{KBSM:Lplus}
\end{subfigure}
\centering
\begin{subfigure}{.25\textwidth}
\centering
$\vcenter{\hbox{\includegraphics[scale = .35]{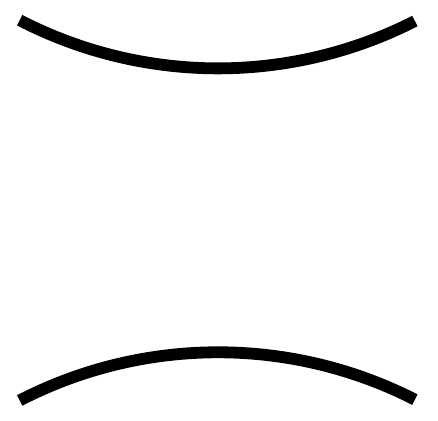}}} $
\caption{$L_0$.} \label{KBSM:Lzero}
\end{subfigure}
\centering
\begin{subfigure}{.25\textwidth}
\centering
$\vcenter{\hbox{\includegraphics[scale = .35]{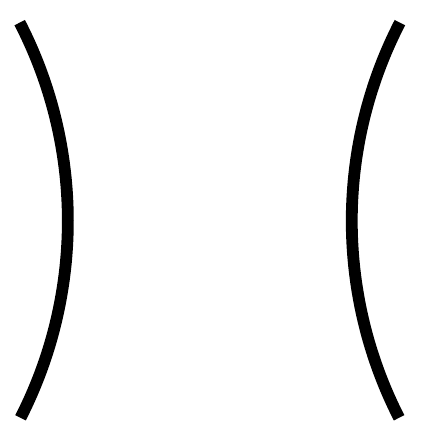}}} $
\caption{$L_{\infty}$.} \label{KBSM:Linfinity}
\end{subfigure}
\caption{The skein triple.} \label{KBSM:skeintriple}
\end{figure}

\begin{theorem}\label{Theorem:freebasis}\cite{Prz1}
Let $F$ be a surface with $\partial F \neq \varnothing$. If $F$ is orientable let $M = F\times I$, otherwise let $M = F \ \hat{\times} \ I$. Let all  the points in $\{x_i\}_1^{2n}$  be marked points that lie on $\partial F \times \{0\}$. Then $\mathcal S_{2, \infty}(M, \{x_i\}_1^{2n}; R, A)$ is a free $R$-module whose basis is composed of relative links in $F$ without trivial components. When $n=0$, the empty link is also a generator.
\end{theorem}

The following corollary to Theorem \ref{Theorem:freebasis} uses the language of relative skein modules that are used in the definition of Gram determinants in knot theory.

\begin{corollary}\label{Corollary:basiselements}\cite{Prz1}

\begin{enumerate}
\item [(a)] $\mathcal{S}_{2,\infty}(D^2 \times I, n) \coloneqq \mathcal{S}_{2,\infty}(D^2 \times I, \{x_i\}_1^{2n}; R, A)$ is a free $R$-module with $C_n = \frac{1}{n+1}\binom{2n}{n}$ basis elements. 

\
    
\item [(b)]$\mathcal{S}_{2,\infty}(\mathit{Ann} \times I, n) \coloneqq \mathcal{S}_{2,\infty}(\mathit{Ann} \times I, \{x_i\}_1^{2n}; R, A)$ where $\{x_i\}_1^{2n}$ are located in the outer boundary component of the annulus is a free $R[z]$-module with $D_n = \binom{2n}{n}$ basis elements, where $z$ denotes the homotopically nontrivial curve in the annulus and $d=-A^2-A^{-2}$ denotes the homotopically trivial curve in the annulus. The basis, denoted by $B_{n, 0}$, is the set of all crossingless connections in the annulus with no trivial components or boundary parallel curves. 

\ 

\item [(c)] $\mathcal{S}_{2,\infty}(Ann_{n, 1} \times I) \coloneqq \mathcal{S}_{2,\infty}(\mathit{Ann} \times I, \{x_i\}_1^{2n}\cup \{y_1, y_2\}; R, A)$, where $\{x_i\}_1^{2n}$ are located in the outer boundary component of the annulus and $\{y_1, y_2 \}$ are located in the inner boundary component of the annulus, is a free $R$-module. The standard basis contains an infinite number of elements described as follows. Elements of the form $az^i$ for $i \geq 0$, where $z$ denotes the boundary parallel curve of the annulus and $a \in \mathcal{A}$. The set $\mathcal{A}$ is a finite collection of crossingless connections with no trivial components or boundary parallel curves where $y_1$ and $y_2$ are connected to each other by a relative link. There are $2 \binom{2n}{n}$ such elements in the set. Let $X$ be a finite collection of crossingless connections between $2n-2$ points on the outer boundary that do not isolate the remaining two points on the outer boundary from a path to the inner boundary points. The rest of the standard basis is from an infinite family of crossingless connections obtained from $X$, where the remaining outer boundary points are connected to the inner boundary points after wrapping around the inner boundary by $\pi k$ for $k \in \mathbb{Z}$; an illustration of such connections can be found in Figure \ref{Figure:AnnDehntwist}. 

\ 

\item[(d)] $\mathcal{S}_{2,\infty}(Mb \ \hat{\times} \  I) \coloneqq \mathcal{S}_{2,\infty}(Mb \ \hat\times \ I, \{x_i\}_1^{2n}; R, A)$ is a free $R$-module. The standard basis contains an infinite number of elements of the form $bz^i$, $bxz^i$ for $i \geq 0$, where $x$ denotes the simple closed curve that intersects the crosscap once, $z$ denotes the boundary parallel curve of the M\"obius band, and $b$ is an element in the set of crossingless connections in the M\"obius band with no trivial components or boundary parallel curves for which the arcs do not intersect the crosscap. The rest of the elements in the standard basis are from a finite number of crossingless connections consisting of a collection of $n-k$ arcs for $0 \leq k <n$ that non-trivially intersect the crosscap. Among the finite collection there are $\binom{2n}{k}$ crossingless connections that intersect the crosscap $n-k$ times.  
\end{enumerate}

\end{corollary}

\subsection{The Gram determinant of type $\boldsymbol{B}$} The origins of the Gram determinant of type $B$ can be found in \cite{MaSa1}. The creation of this type of Gram determinants comes from the Blob algebra (see \cite{MaSa2}) that is associated to the transfer matrix formulation of statistical mechanics on arbitrary lattices. 
The knot theoretic interest originated from the work of R. Simion \cite{Sim} on chromatic joins when Q. Chen and J. H. Przytycki in \cite{Ch-P2} showed a connection to the Gram matrix of the Temperley-Lieb algebra and the matrix of chromatic joins. Furthermore, Q. Chen and J. H. Przytycki in \cite{CP, Che} used skein modules, the Jones-Wenzl idempotents, and Chebyshev polynomials to prove a closed formula for type $B$. The definition we give will use the language of skein modules.

\begin{definition} 
Let $B_{n, 0} = \{b_1, b_2, \ldots, b_{\binom{2n}{n}} \}$ be the set of all diagrams of crossingless connections between $2n$ marked points on the outer boundary of $\mathit{Ann} \times \{0\}$ in $\mathit{Ann} \times I$. Define the bilinear form $\langle \ \ , \ \rangle_{B}$ for type $B$ as follows:
$$\langle \ \ , \ \rangle_{B} : \mathcal{S}_{2,\infty}(\mathit{Ann} \times I, \{x_i\}_1^{2n}; R, A) \times \mathcal{S}_{2,\infty}(\mathit{Ann} \times I, \{x_i\}_1^{2n}; R, A) \longrightarrow R[d,z].$$

Given $b_i, b_j \in B_{n, 0}$, we glue $b_i$ with the inversion of $b_j$ along the marked outer boundary, respecting the labels of the marked points. The result is an element in $\mathit{Ann}\times I$ containing only disjoint simple closed curves which are either homotopically non-trivial (denoted by $z$), or null homotopic (denoted by $d$). Then, $\langle b_i , b_j\rangle_{B} \coloneqq z^k d^m$, where $k$ and $m$ denote the number of these curves, respectively.

The Gram matrix of type $B$ is defined as $G_n^{B} = (\langle b_i , b_j\rangle_{B}) _{1 \leq i, j \leq \binom{2n}{n}}$, and its determinant $D_n^B$ is called the Gram determinant of type $B$.
\end{definition}

\begin{example} An example of the bilinear form on two elements in $B_{5, 0}$ is given below.
$$\left\langle \vcenter{\hbox{\begin{overpic}[scale = .2]{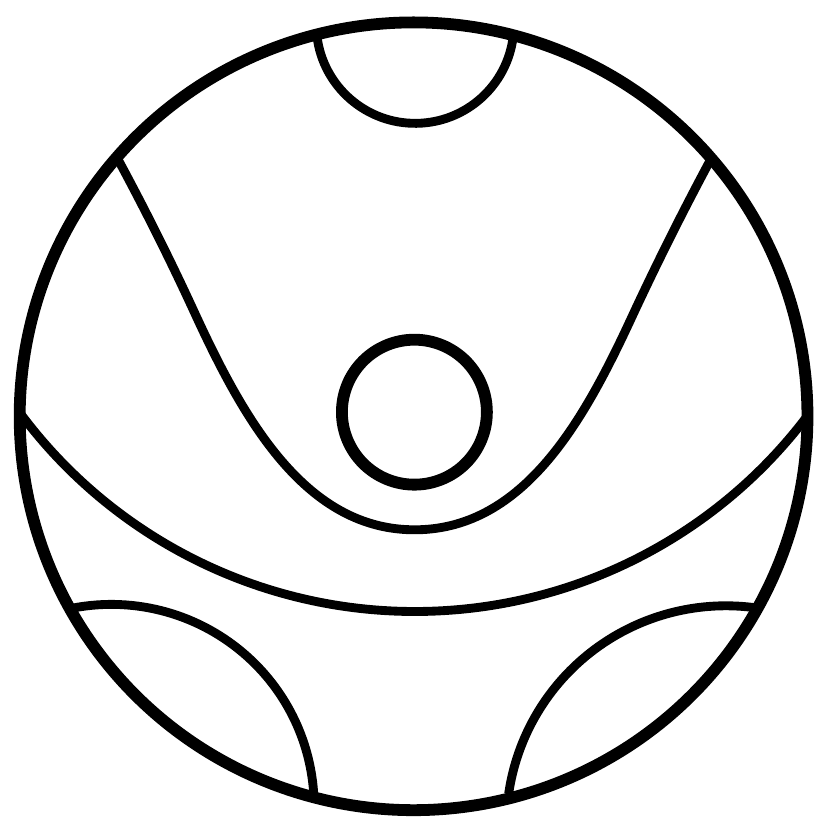}
\end{overpic} }}, \vcenter{\hbox{\begin{overpic}[scale = .2]{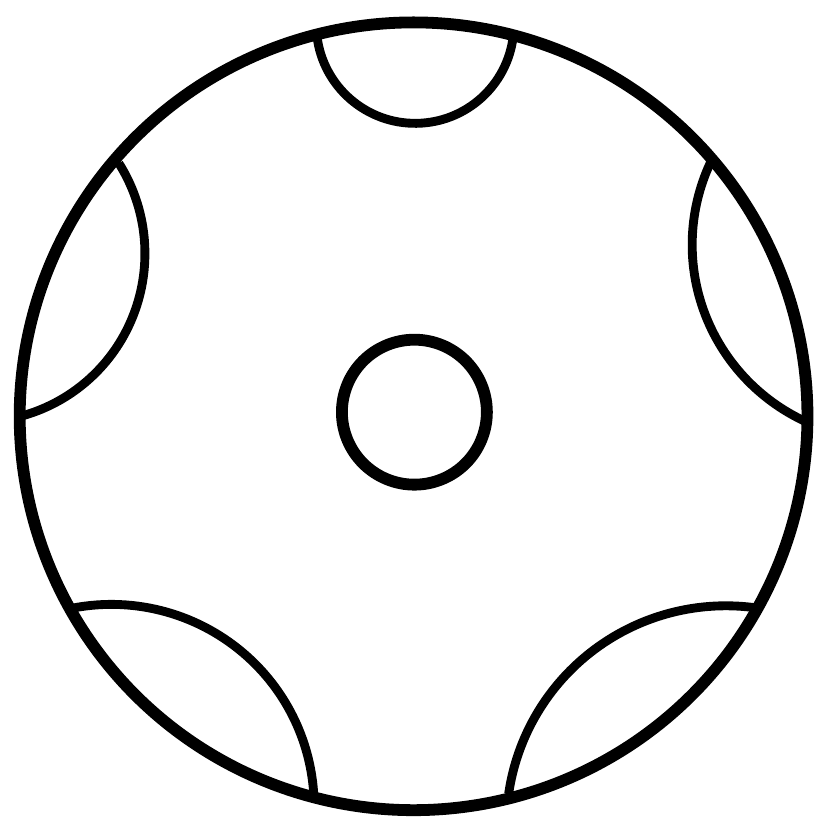}
\end{overpic} }} \right\rangle_{B} = \vcenter{\hbox{\begin{overpic}[scale = .2]{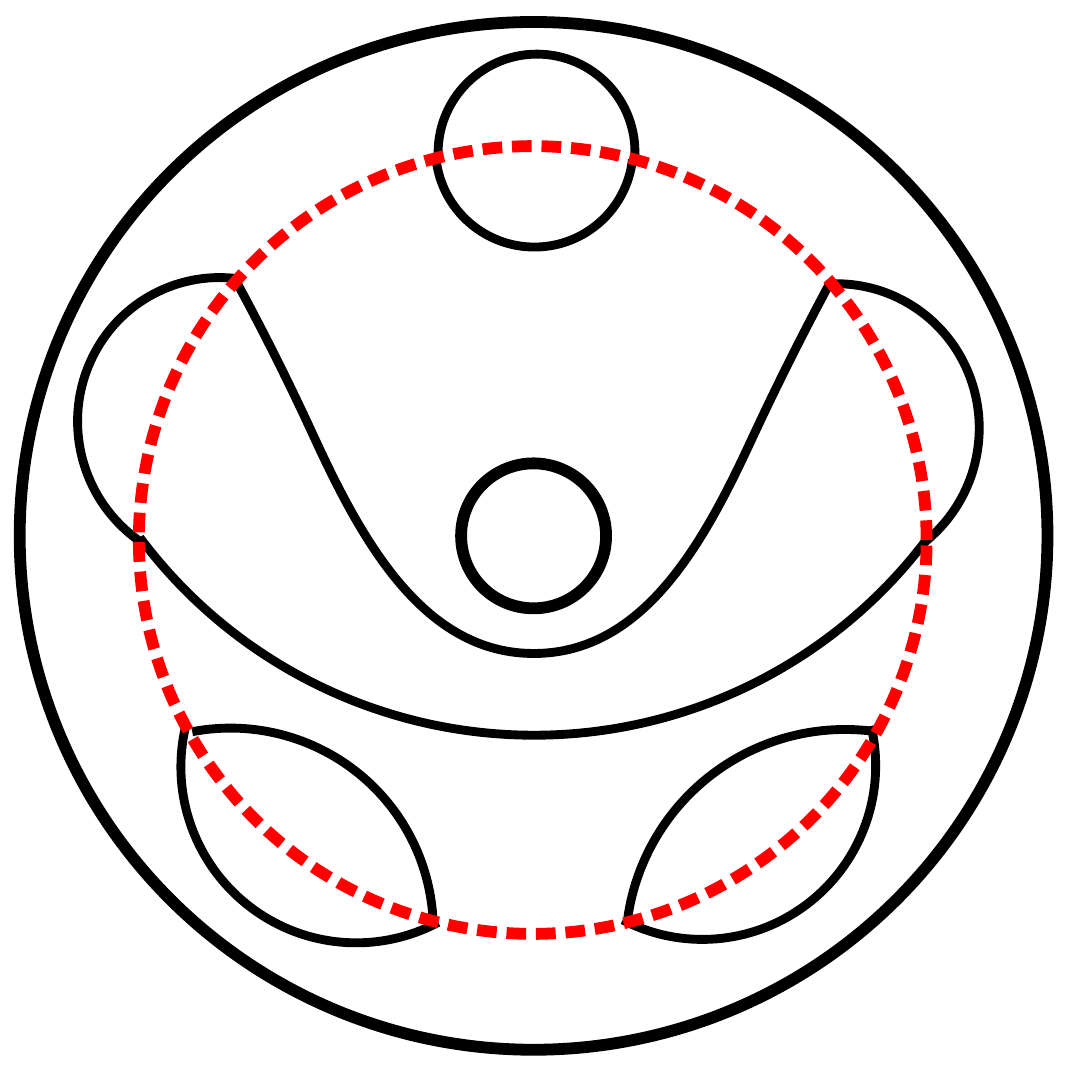}
\end{overpic} }} = d^4. $$
\end{example}

\begin{theorem}\cite{CP}\cite{MaSa1} \label{Theorem:TypeB}
Let $R = \mathbb{Z}[A^{\pm1}]$. Then 
$$D_n^{B}=  \prod \limits_{i=1}^n(T_i(d)^2-z^2)^{\binom{2n}{n-i}},$$
where $T_i(d)$ is the Chebyshev polynomial of the first kind recursively defined by $T_{n+1}(d)=d T_{n}(d)- T_{n-1}(d)$, with initial conditions $T_{0}(d)=2$ and $T_{1}(d)=d$, where $d = -A^2 -A^{-2}$.
\end{theorem}

Q. Chen and J. H. Przytycki's proof of a closed formula for type $B$ involves the creation of a linear map on the basis $B_{n, 0}$ that uses the lollipop method to decorate the inner boundary component with the Jones-Wenzl idempotent, then proving that the image of the linear map under the basis is a subspace of dimension $\binom{2n}{n} - \binom{2n}{n-k}$.

\subsection{The Gram determinant of type $\boldsymbol{Mb}$}

J\'ozef H. Przytycki constructed the notion of the Gram determinant of type $Mb$ in 2008. This originates from the study of crossingless connections on a M\"obius band. Here the bilinear form is defined through the identification of two M\"obius bands along their boundaries. This determinant is given in Definition \ref{TypeMB}.

\begin{definition}\label{TypeMB} 
Let $\mathit{Mb}_n = \{ m_1, \dots, m_{\sum_{k=0}^{n} \binom{2n}{k}}\}$  be the set of all diagrams of crossingless connections between $2n$ marked points on the boundary of the M\"obius band $\mathit{Mb} \ \hat{\times} \ \{0\}$ in $\mathit{Mb} \  \hat{\times} \ I$. 
Define a bilinear form $\langle \ , \ \rangle_{\mathit{Mb}}$ on the elements of $\mathit{Mb}_n$ as follows:
$$ \langle \ , \ \rangle_{\mathit{Mb}} : \mathcal{S}_{2,\infty}(\mathit{Mb}\  \hat \times \ I, \{x_i\}_1^{2n}) \times \mathcal{S}_{2,\infty}(\mathit{Mb}\ \hat \times \  I, \{x_i\}_1^{2n}) \longrightarrow \mathbb{Z}[d,w,x,y,z].$$

Given $m_i, m_j \in \textit{Mb}_n$, identify the boundary component of $m_i$ with that of the inversion of $m_j$, respecting the labels of the marked points. The result is an element in $Kb \ \hat{\times} \ I$ containing only disjoint simple closed curves. The five homotopically distinct simple closed curves in the Klein bottle, including the homotopically trivial curve, are denoted by $x, y, z, w, d$ as illustrated in Figure \ref{Klein}. Then, $\langle m_i , m_j\rangle_{\mathit{Mb}} \coloneqq  d^mx^ny^kz^lw^h$ where $m,n,k,l$ and $h$ denote the number of these curves, respectively.

The Gram matrix of type $\mathit{Mb}$ is defined as $\ G_n^{\mathit{Mb}} = (\langle m_i , m_j\rangle_{\mathit{Mb}})_{1 \leq i, j \leq \sum_{k=0}^{n} \binom{2n}{k}}$ and its determinant $D_n^{\mathit{Mb}}$ is called the Gram determinant of type $\mathit{Mb}$.

\begin{figure}[ht]
    \centering
    $$ \vcenter{\hbox{
\begin{overpic}[scale = 3]{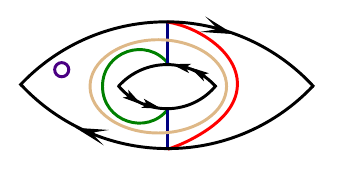}
\put(37, 80){$d$}
\put(70, 70){$z$}
\put(105, 50){$x$}
\put(150, 100){$w$}
\put(210, 70){$y$}
\end{overpic} }}$$
    \caption{Klein bottle and its five homotopically distinct simple closed curves.}
    \label{Klein}
\end{figure}
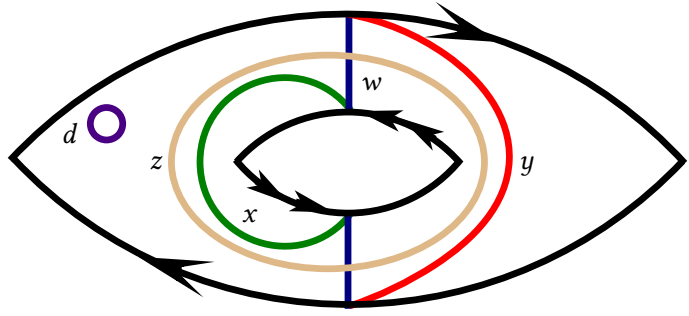

\end{definition}

\begin{example}\label{example:mbn} An example of the bilinear form on two elements in $Mb_4$ is given below, where $\vcenter{\hbox{\begin{overpic}[scale = .2]{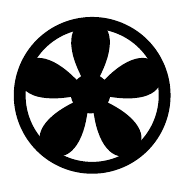}
\end{overpic}}}$ denotes a crosscap usually denoted by $\vcenter{\hbox{\begin{overpic}[scale = 1]{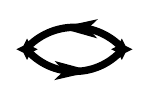} \end{overpic}}}.$
$$\left\langle \vcenter{\hbox{\begin{overpic}[scale = .2]{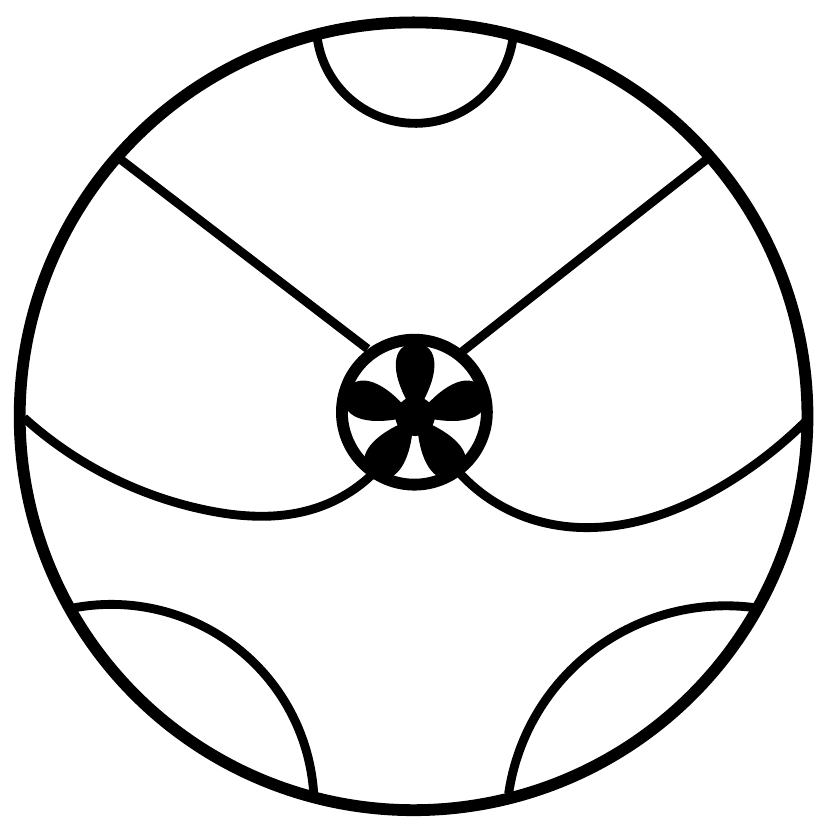}
\end{overpic}}}, \vcenter{\hbox{\begin{overpic}[scale = .2]{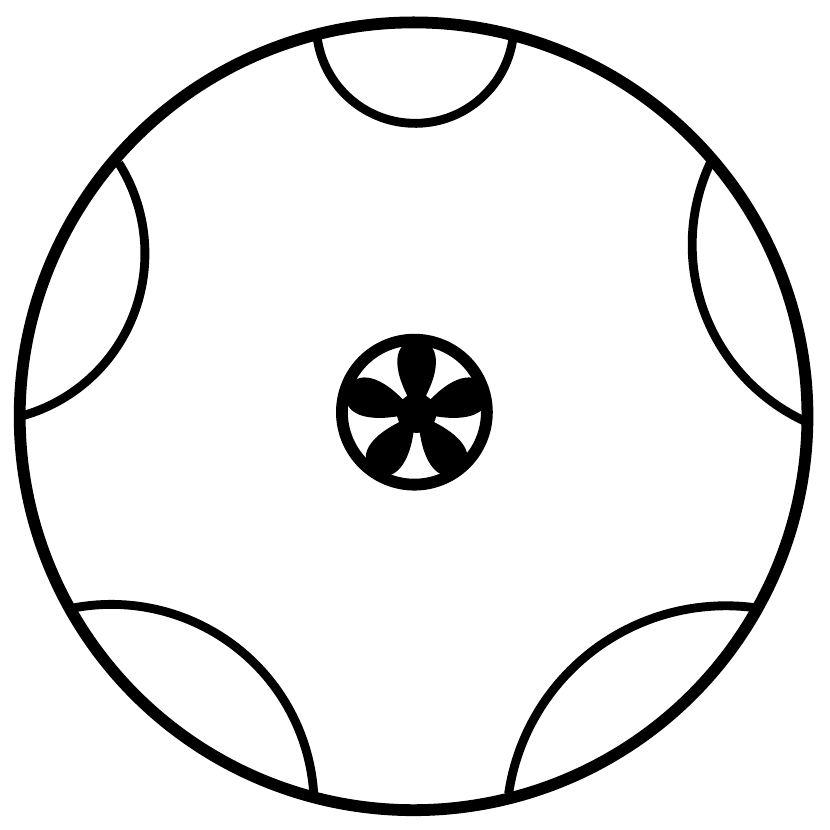}
\end{overpic}}} \right\rangle_{B} = \vcenter{\hbox{\begin{overpic}[scale = .2]{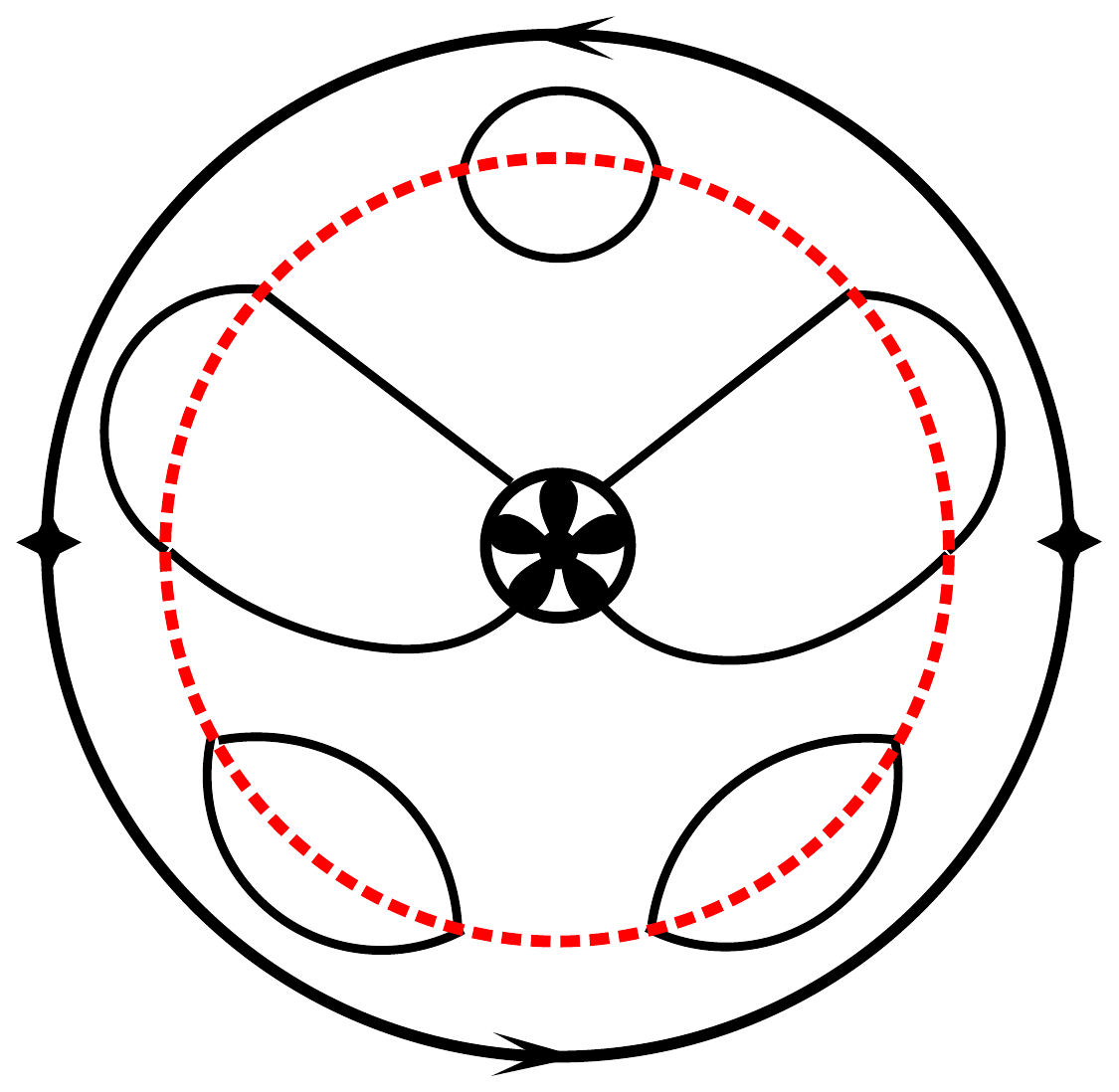}
\end{overpic} }} = d^4. $$
\end{example}

Q. Chen conjectured the following result for the Gram determinant of type $\mathit{Mb}$. Some work supporting this conjecture can be found in \cite{BIMP, BIMP2, PBIMW}. %\cite{CP}

\begin{conjecture}[Chen] \label{Qi}\

Let $R = \mathbb{Z}[A^{\pm 1},w,x,y,z].$ Then
the Gram determinant of type {\it Mb} for $n \geq 1$, denoted by $D_n^{(\mathit{Mb})}$, is:
 \begin{eqnarray*}
D^{(\mathit{Mb})}_n(d,w,x,y,z) &=& \prod_{k=1}^n (T_k(d)+(-1)^kz)^{\binom{2n}{n-k} } \\
& & \prod\limits_{ \substack{k=1 \\ k\text{ odd }}}^n
%_{k=1 \atop k \text{ odd }}^n 
\left((T_k(d) - (-1)^k z)T_k(w) -2xy\right)^{\binom{2n}{n-k}} \\
& & \prod\limits_{ \substack{k=1 \\ k\text{ even }}}^n
%_{ k=1 \atop k \text{ even }}^n 
\left((T_k(d) - (-1)^kz)T_k(w)-2(2-z)\right)^{\binom{2n}{n-k}} \\
& & \prod_{i=1}^{n-1} D_{n,i},
\end{eqnarray*}

where $D_{n,i} = \prod\limits_{k=1+i}^n (T_{2k}(d)-2)^{\binom{2n}{n-k}}$, and $i$ represents the number of curves passing through the crosscap.
\end{conjecture}

\begin{proposition}\cite{BIMP2}\label{GRAMMB:Prop1}
$D^{\mathit{Mb}}_n$ is divisible by $(d-z)^{\binom{2n}{n-1}}$.
\end{proposition}

\begin{proposition}\label{GRAMMB:Prop2}\cite{BIMP}
$D^{\mathit{Mb}}_n$  is divisible by $(w(d+z)-2xy)^{\binom{2n}{n-1}}$.
\end{proposition}

\begin{proposition}\label{GRAMB}\cite{PBIMW}
$D_n^{\mathit{Mb}}$ is divisible by $((d^2-2-z)(w^2-2)-2(2-z))^{\binom{2n}{n-2}}$.
\end{proposition}

\section{The Gram determinant of type $\boldsymbol{(Mb)_1}$}\label{GramMB1}
The basis of the relative Kauffman bracket skein module of the twisted $I$-bundle of the M\"obius band, as described in Corollary \ref{Corollary:basiselements}, is infinite. If we restrict to only basis elements with no $z$ and $x$ curves, then we obtain a finite sub-collection of this basis. The Gram matrix of type $\mathit{Mb}$ was created by this finite sub-collection. However, this construction carries a few disadvantages. First, the number of elements increases exponentially as $n$ increases. In particular, computing the determinant for $n \geq 5$ has not been achieved due to the size of the matrices. Second, the lollipop technique of decorating the crosscap with the Jones-Wenzl idempotent is no longer as straightforward as it was in the case of type $B$. In fact, attempts to using this technique have not yet been rigorously successful. In this section we propose a new Gram determinant from a sub-collection of $Mb_n$, give direct connections to a new type $B$, and explain how the lollipop method can be applied to this case and where it falls short for the case of type $\mathit{Mb}$.

\begin{definition}\label{TypeMB1} 
Let $(\mathit{Mb}_{n})_1 = \mathit{Mb}_{n,0} \cup \mathit{Mb}_{n,1}$ where $\mathit{Mb}_{n,0} = \{ m_1, \dots, m_{\binom{2n}{n}}\}$ is the set of all diagrams of crossingless connections between $2n$ marked points on the boundary of $\mathit{Mb} \ \hat{\times} \ \{0\}$ whose arcs do not intersect the crosscap and $\mathit{Mb}_{n,1} = \{ m_1, \dots, m_{\binom{2n}{n-1}}\}$ is the set of all diagrams of crossingless connections between $2n$ marked points on the boundary of  $\mathit{Mb} \ \hat{\times} \ \{0\}$ with exactly one curve intersecting the crosscap. 
Define a bilinear form $\langle \ , \ \rangle_{\mathit{Mb}}$ on the elements of $(\mathit{Mb}_{n})_1$ by using the same bilinear form as type $\mathit{Mb}$, as follows: $$ \langle \ , \ \rangle_{\mathit{Mb}} : \mathcal{S}_{2,\infty}(\mathit{Mb}\  \hat \times \ I, \{x_i\}_1^{2n}) \times \mathcal{S}_{2,\infty}(\mathit{Mb}\ \hat \times \  I, \{x_i\}_1^{2n}) \longrightarrow \mathbb{Z}[d,w,x,y,z].$$

Given $m_i, m_j \in (\mathit{Mb}_{n})_1$, identify the boundary component of $m_i$ with that of the inversion of $m_j$, respecting the labels of the marked points. The result is an element in $Kb \ \hat{\times} \ I$ containing only disjoint simple closed curves. Then $\langle m_i , m_j\rangle_{\mathit{Mb}} :=  d^mx^ny^kz^lw^h$ where $m,n,k,l$ and $h$ denote the number of these curves, respectively.

The Gram matrix of type $(\mathit{Mb})_1$ is defined as $\ G_n^{(\mathit{Mb})_1} = (\langle m_i , m_j\rangle_{\mathit{Mb}})_{1 \leq i, j \leq \binom{2n}{n-1} +\binom{2n}{n}}$ and its determinant $D_n^{(\mathit{Mb})_1}$ is called the Gram determinant of type $(\mathit{Mb})_1$.
\end{definition}

In Example \ref{exampleMB1} we show the smallest Gram matrix of type $(Mb)_1$ that differs from type $Mb$.
\begin{example}\label{exampleMB1} 
For $n=2$, the set $(Mb_n)_1$ is given by
$$\left\{ \vcenter{\hbox{\begin{overpic}[scale = .15]{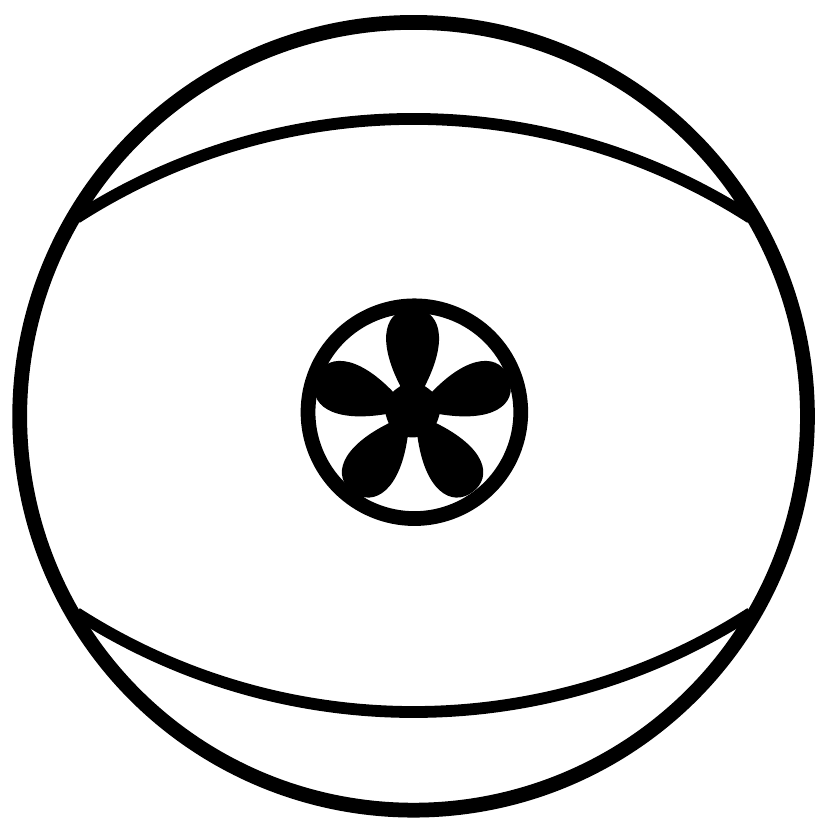}
\end{overpic} }}, \vcenter{\hbox{\begin{overpic}[scale = .15]{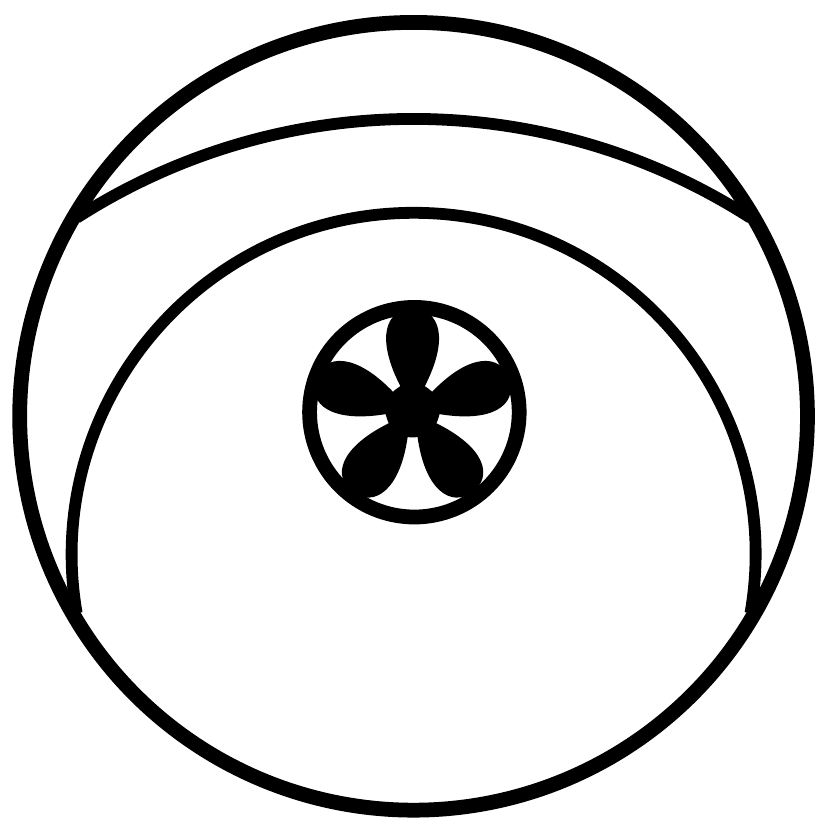}
\end{overpic} }}, \vcenter{\hbox{\begin{overpic}[scale = .15]{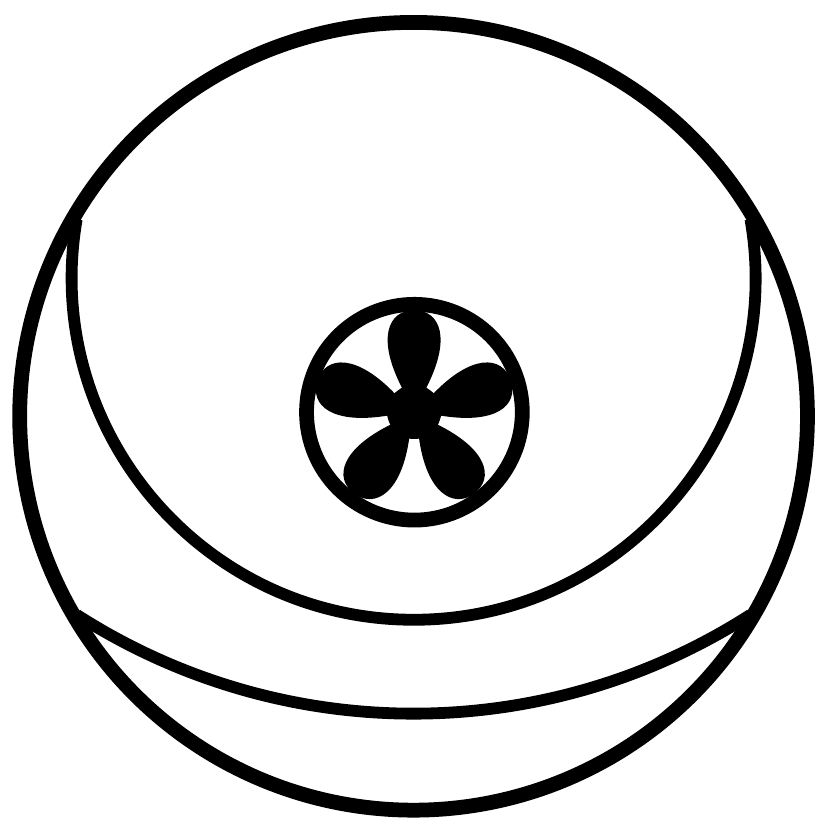}
\end{overpic} }}, \vcenter{\hbox{\begin{overpic}[scale = .15]{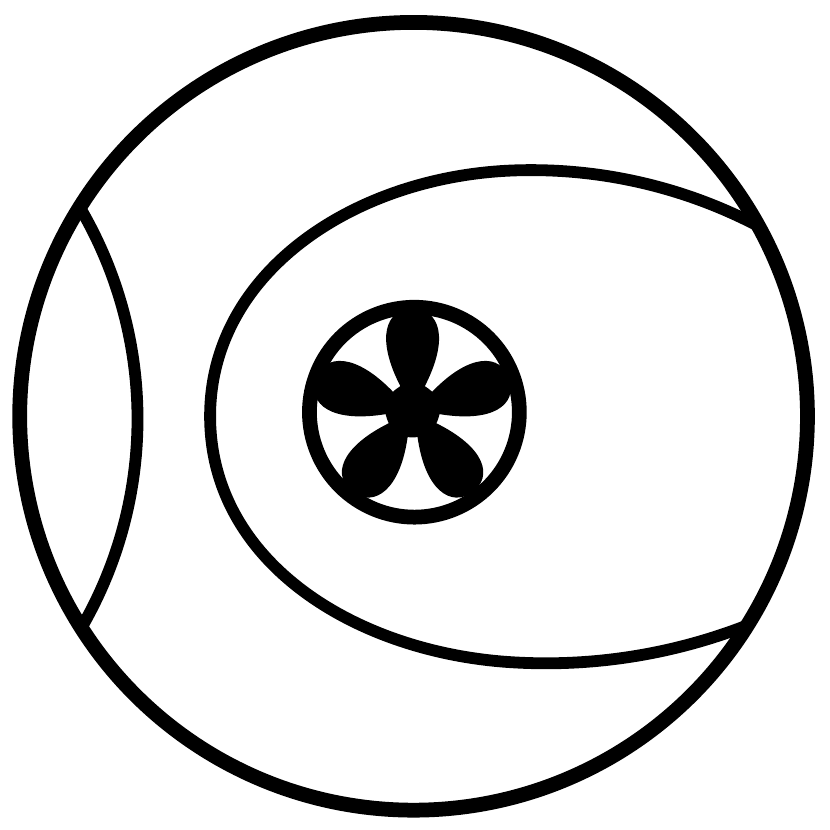}
\end{overpic} }}, \vcenter{\hbox{\begin{overpic}[scale = .15]{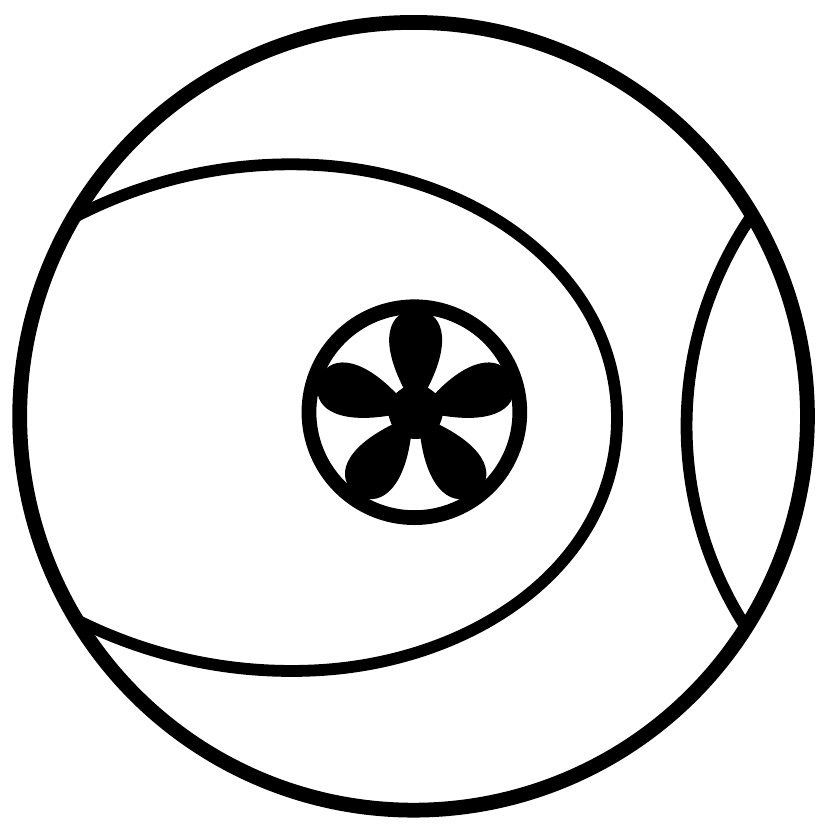}
\end{overpic} }}, \vcenter{\hbox{\begin{overpic}[scale = .15]{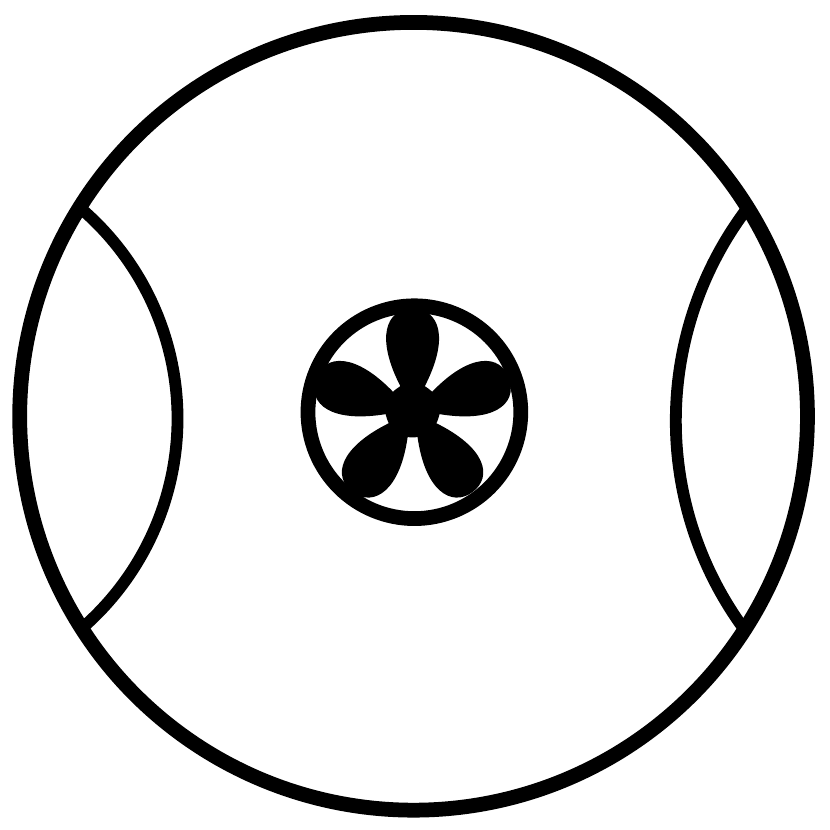}
\end{overpic} }}, \vcenter{\hbox{\begin{overpic}[scale = .15]{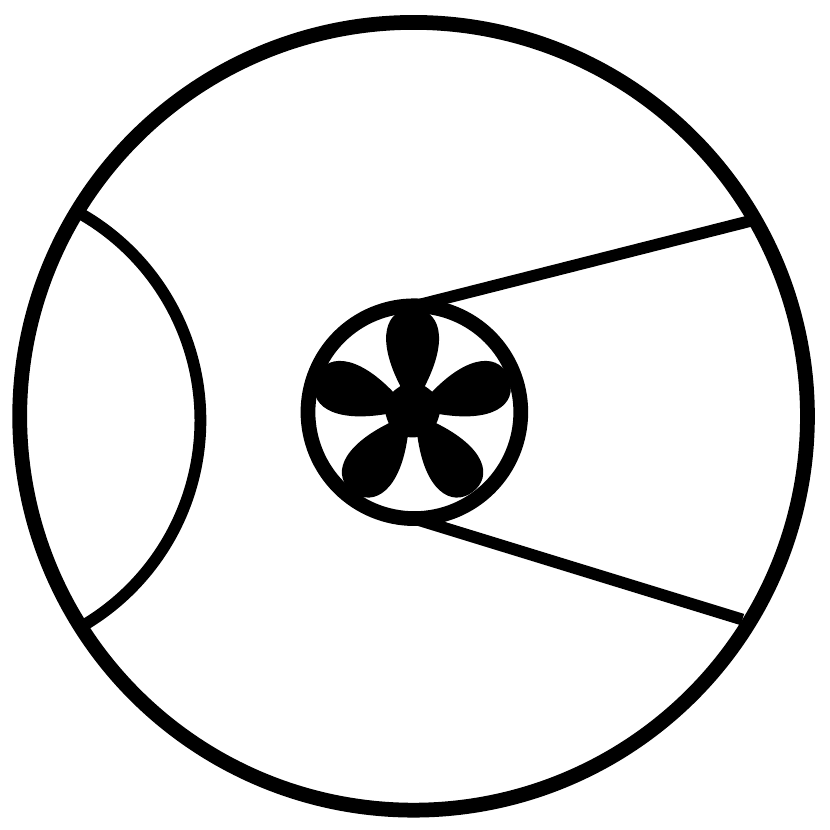}
\end{overpic} }}, \vcenter{\hbox{\begin{overpic}[scale = .15]{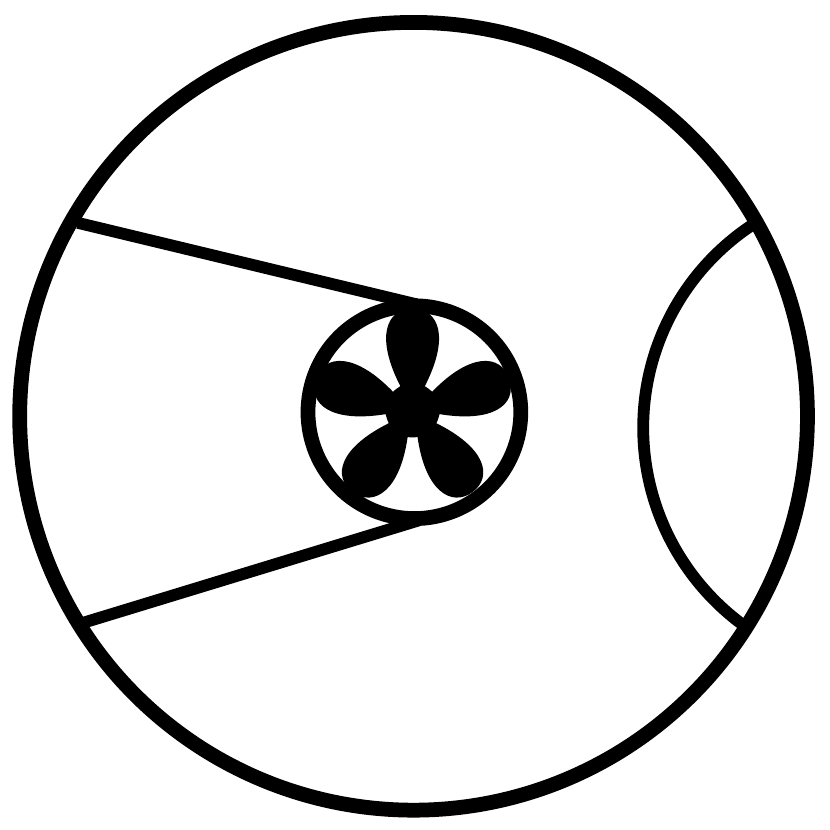}
\end{overpic} }}, \vcenter{\hbox{\begin{overpic}[scale = .15]{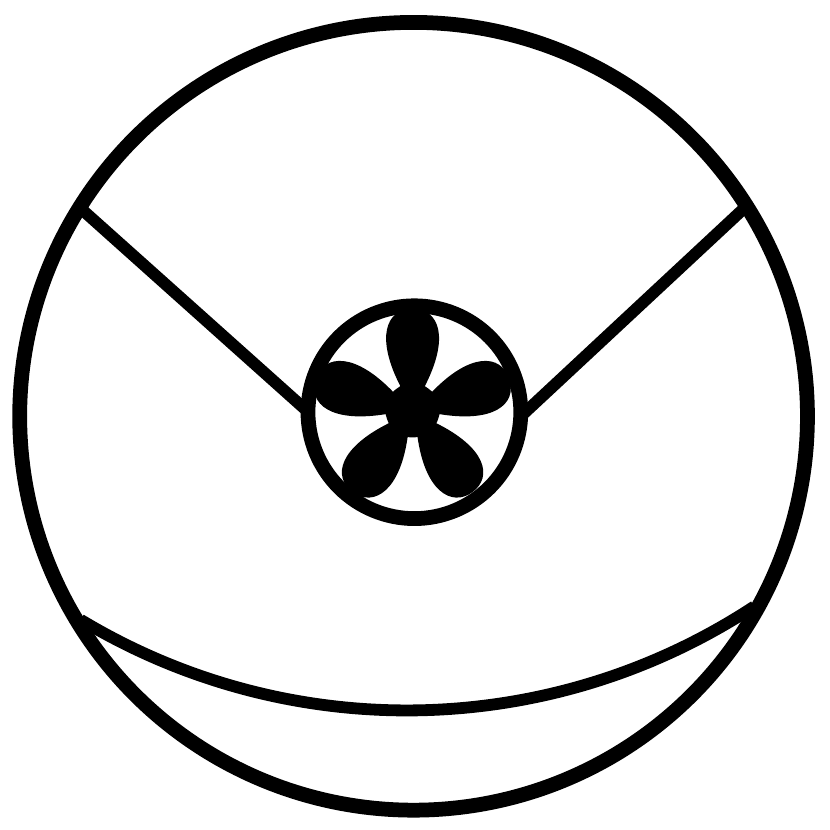}
\end{overpic} }}, \vcenter{\hbox{\begin{overpic}[scale = .15]{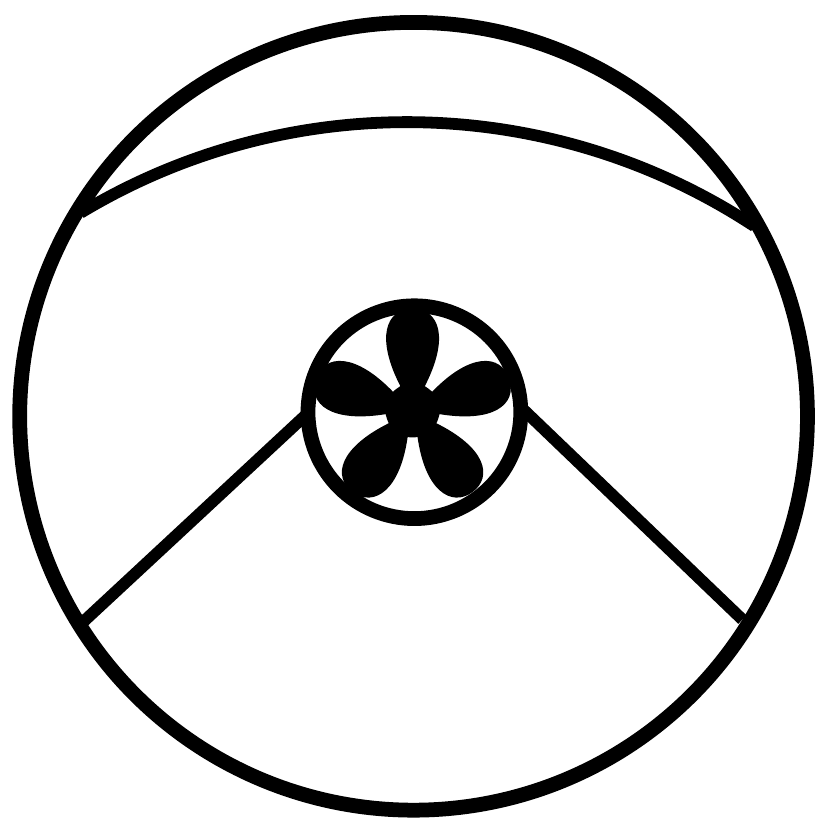}
\end{overpic} }} \right\}, $$
and the Gram matrix is shown in Table \ref{Grammatrixn2}. \\ 

{\centering
\resizebox{\columnwidth}{!}{%
\begin{tabular}{ c|||c|c|c|c|c|c||c|c|c|c}
 $\langle \ , \ \rangle$ & $\vcenter{\hbox{\includegraphics[scale = .13, height = 0.7cm]{GMB2_2.pdf}}}$ 
& $\vcenter{\hbox{\includegraphics[scale = .13,height = 0.7cm]{GMB2_1.pdf}}}$
& $\vcenter{\hbox{\includegraphics[scale = .13,height = 0.7cm]{GMB2_3.pdf}}}$
& $\vcenter{\hbox{\includegraphics[scale = .13,height = 0.7cm]{GMB2_4.pdf}}}$
& $\vcenter{\hbox{\includegraphics[scale = .13,height = 0.7cm]{GMB2_6.pdf}}}$
& $\vcenter{\hbox{\includegraphics[scale = .13,height = 0.7cm]{GMB2_5.pdf}}}$
& $\vcenter{\hbox{\includegraphics[scale = .13,height = 0.7cm]{GMB2_10.pdf}}}$
& $\vcenter{\hbox{\includegraphics[scale = .13,height = 0.7cm]{GMB2_8.pdf}}}$
& $\vcenter{\hbox{\includegraphics[scale = .13,height = 0.7cm]{GMB2_9.pdf}}}$
& $\vcenter{\hbox{\includegraphics[scale = .13,height = 0.7cm]{GMB2_7.pdf}}}$ \\
\hline \hline \hline 	
$\vcenter{\hbox{\includegraphics[scale = .13, height = 0.7cm]{GMB2_2.pdf}}}$        
& $ d^2$ & $ dz$ & $z^2$ & $z$ & $d$ & $z$ & $dy$ & $y$ & $yz$ & $ y$  \\ \hline 
$\vcenter{\hbox{\includegraphics[scale = .13, height = 0.7cm]{GMB2_1.pdf}}}$          
& $ dz$ & $ d^2$ & $dz$ & $d$ & $z$ & $d$ & $dy$ & $y$ & $dy$ & $ y$  \\ \hline 
$\vcenter{\hbox{\includegraphics[scale = .13, height = 0.7cm]{GMB2_3.pdf}}}$ 
& $ z^2$ & $ dz$ & $d^2$ & $z$ & $d$ & $z$ & $yz$ & $y$ & $dy$ & $ y$  \\ \hline 
$\vcenter{\hbox{\includegraphics[scale = .13, height = 0.7cm]{GMB2_4.pdf}}}$  
& $ z$ & $ d$ & $z$ & $d^2$ & $dz$ & $z^2$ & $y$ & $yz$ & $y$ & $ dy$ \\ \hline 
$\vcenter{\hbox{\includegraphics[scale = .13, height = 0.7cm]{GMB2_6.pdf}}}$  
& $ d$ & $ z$ & $d$ & $dz$ & $d^2$ & $dz$ & $y$ & $dy$ & $y$ & $ dy$  \\ \hline 
$\vcenter{\hbox{\includegraphics[scale = .13, height = 0.7cm]{GMB2_5.pdf}}}$  
& $ z$ & $ d$ & $z$ & $z^2$ & $dz$ & $d^2$ & $y$ & $dy$ & $y$ & $ yz$  \\ \hline \hline
$\vcenter{\hbox{\includegraphics[scale = .13, height = 0.7cm]{GMB2_10.pdf}}}$  
& $ dx$ & $ dx$ & $xz$ & $x$ & $x$ & $x$ & $dw$ & $w$ & $xy$ & $w$ \\ \hline 
$\vcenter{\hbox{\includegraphics[scale = .13, height = 0.7cm]{GMB2_8.pdf}}}$  
& $x$ & $ x$ & $x$ & $xz$ & $dx$ & $dx$ & $w$ & $dw$ & $w$ & $ xy$  \\ \hline 
$\vcenter{\hbox{\includegraphics[scale = .13, height = 0.7cm]{GMB2_9.pdf}}}$  
& $ xz$ & $ dx$ & $dx$ & $x$ & $x$ & $x$ & $xy$ & $w$ & $dw$ & $ w$  \\ \hline 
$\vcenter{\hbox{\includegraphics[scale = .13, height = 0.7cm]{GMB2_7.pdf}}}$  
& $ x$ & $ x$ & $x$ & $dx$ & $dx$ & $xz$ & $w$ & $xy$ & $w$ & $ dw$ 
\end{tabular}}}
\captionof{table}{The Gram matrix $G_{2}^{(Mb)_1} .$}\label{Grammatrixn2}\label{table}

\begin{eqnarray*}
D_2^{(\mathit{Mb})_1} &=& (-2 + d) d^2 (2 + d) (d - z)^4 (-2 + d^2 - z) (-2 + d^2 + z) (-d w + 
   2 x y - w z)^4 \\
   &=& (d-z)^4((d^2-2)+z) ((d+z)w-2xy)^4((d^2-2)-z)(d^2(d^2-4)+2-2) \\
   &=& (T_1(d)-z)^4(T_2(d)^2-z^2)((d+z)w-2xy)^4(T_4(d)-2).
\end{eqnarray*}
\end{example}

The following proposition is a direct result from a proof given in \cite{BIMP} for type $Mb$.

\begin{proposition}\cite{BIMP}\label{GRAMMB1:Prop2}
$D^{(\mathit{Mb})_1}_n$  is divisible by $(w(d+z)-2xy)^{\binom{2n}{n-1}}$.
\end{proposition}

Consider an annulus with $2n$ points on the outer boundary and $2$ points on the inner boundary. Fix a crossingless connection between $2n-2$ points on the outer boundary that do not isolate the remaining two points on the outer boundary from a path to the inner boundary points. Then there are an infinite number of ways to connect the remaining outer boundary points to the inner boundary points without introducing a crossing. For example, if we start by connecting two arcs, each on the outer boundary, then we may wrap the arcs around the annular boundary by $\pi k$ where $k \in \mathbb{Z}$ before connecting to the inner boundary points. An illustration is given in Figure \ref{Figure:AnnDehntwist}. 

\begin{figure}[ht]
    \centering
    \begin{subfigure}{.24\textwidth}
\centering
    $$\vcenter{\hbox{\begin{overpic}[scale = .5]{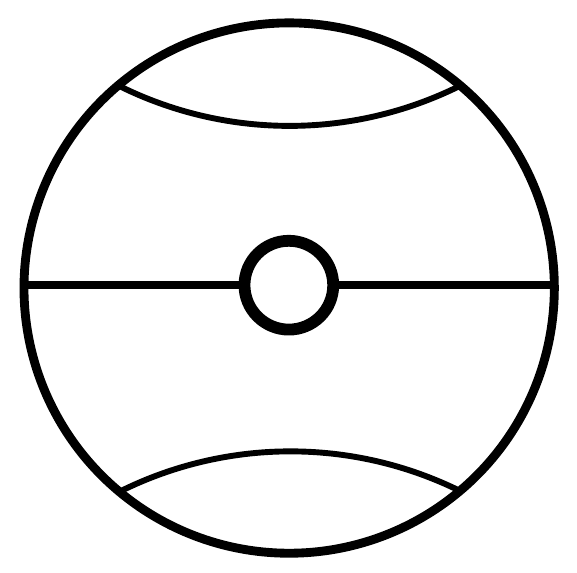}
\end{overpic} }}$$
    \caption{$k=0$.}
     \label{fig:whitemarkerexampleboundary1}
\end{subfigure}
\begin{subfigure}{.24\textwidth}
\centering
    $$\vcenter{\hbox{\begin{overpic}[scale = .5]{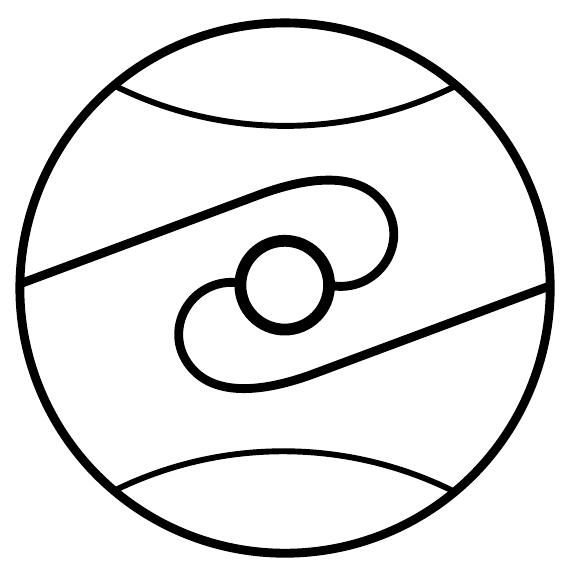}
\end{overpic} }}$$
    \caption{$k=-1$.}
     \label{fig:AnnK1}
\end{subfigure}
\begin{subfigure}{.24\textwidth}
\centering
    $$\vcenter{\hbox{\begin{overpic}[scale = .5]{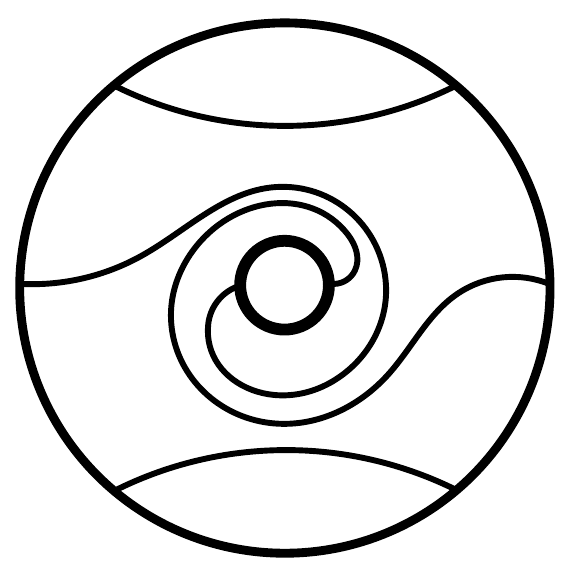}
\end{overpic} }}$$
    \caption{$k=-2$.}
     \label{fig:AnnK2}
\end{subfigure}
\begin{subfigure}{.24\textwidth}
\centering
    $$\vcenter{\hbox{\begin{overpic}[scale = .5]{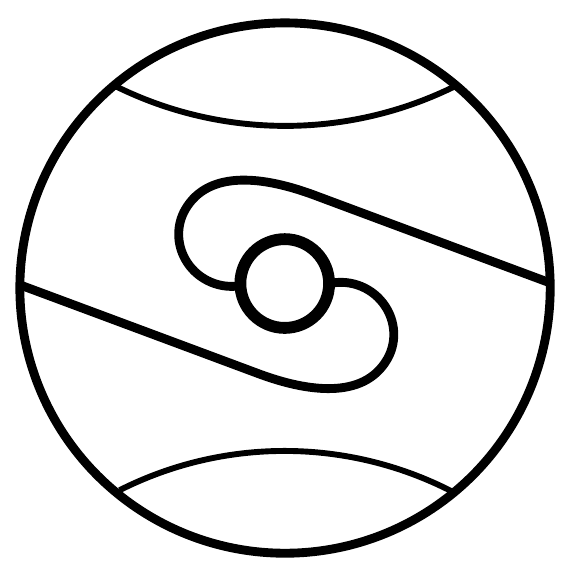}
\end{overpic} }}$$
    \caption{$k=1$.}
     \label{fig:AnnKn1}
\end{subfigure}
    \caption{Examples of distinct crossingless connections from an infinite family where two arcs attached to the inner and outer boundary wrap around the annular boundary $k/2$ times for $k \in \mathbb{Z}$.}
    \label{Figure:AnnDehntwist}
\end{figure}

Fix a line segment connected between the $1^{st}$ and $2n^{th}$ marked point of the outer boundary and between the two inner boundary points. Call this segment the \textbf{lollipop}. As before, consider a fixed crossingless connection between $2n-2$ points on the outer boundary that do not isolate the remaining four points. If the two arcs bounding the remaining four marked points are not allowed to intersect the lollipop, then there is only one way, up to isotopy, to connect the arcs from the outer boundary points to the inner boundary points, without introducing crossings. 

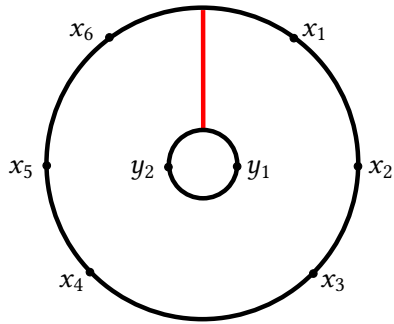
\begin{figure}[ht]
    \centering
    $$ \vcenter{\hbox{
\begin{overpic}[scale = .5]{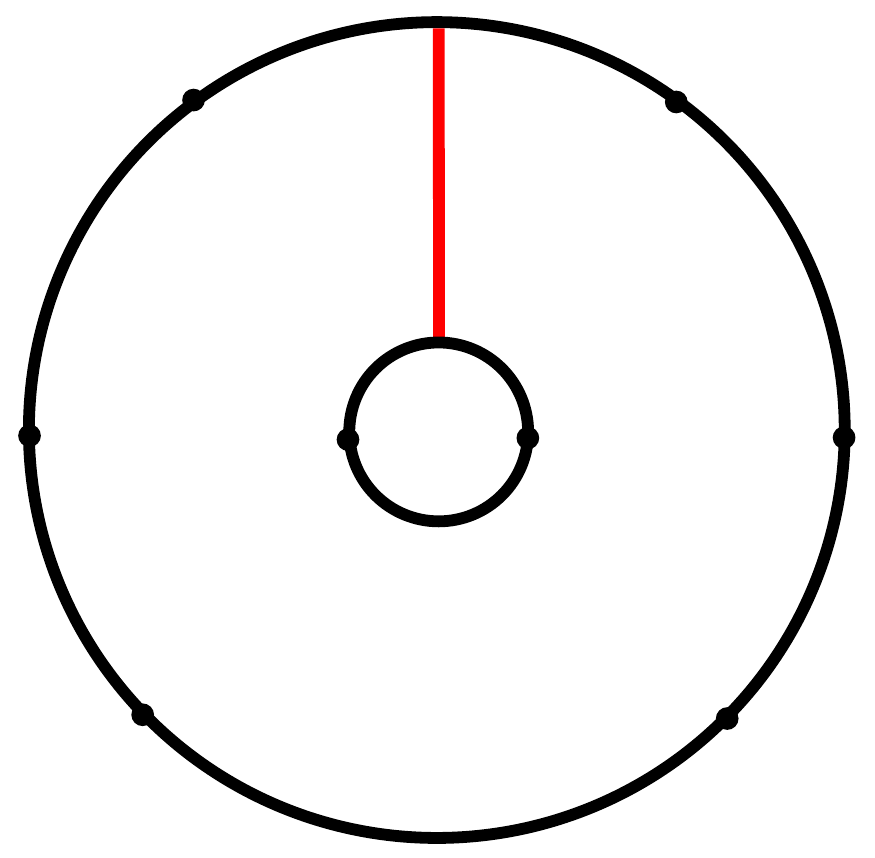}
\put(101,110){$x_1$}
\put(126,58){$x_2$}
\put(108,15){$x_3$}
\put(9,15){$x_4$}
\put(-10, 58){$x_5$}
\put(13,110){$x_6$}
\put(36, 58){$y_2$}
\put(80, 58){$y_1$}
\end{overpic} }}$$
\caption{An illustration of a lollipop in an annulus with $6$ marked points on the outer boundary and two marked points on the inner boundary.}
\label{fig:lollipop}
\end{figure}

The following lemma is a generalization of the children playing a game proof given in \cite{Prz1} and explained in \cite{PBIMW}. Also, one can find a detailed proof of the lemma in \cite{Iba}.
\begin{lemma}\cite{PBIMW, Prz1}
Consider an annulus with $2n$ marked points in the outer boundary, $2$ marked points in the inner boundary, and suppose it contains a lollipop. Let $B_{n,1} = \{ b'_1, \cdots, b'_N \}$ be the set of crossingless connections, up to isotopy, with the following properties:
\begin{enumerate}
    \item There are $n-1$ arcs connected to the marked points in the outer boundary.
    \item These $n-1$ arcs do not isolate the remaining two marked points in the outer boundary from the marked points in the inner boundary.
    \item The two remaining arcs are disjoint from the lollipop and
    \item each are connected to the inner and outer boundary components of the annulus, respectively.
\end{enumerate}
Then $N = \binom{2n}{n-1}$ and there is a one-to-one correspondence between $B_{n, 1}$ and $Mb_{n,1}$.
\end{lemma}
\begin{definition} Consider $\mathcal{S}_{2,\infty}(\mathit{Ann} \times I, \{x_i\}_1^{2n}) \oplus \mathcal{S}_{2,\infty}(\mathit{Ann} \times I, \{x_i\}_1^{2n} \cup \{ y^{inn}_1, y^{inn}_2\})$ and its submodule, $\mathcal{S}((B_n)_1) = \mathcal{S}(B_{n, 0}) \oplus \mathcal{S}(B_{n, 1})$, where the skein modules $\mathcal{S}(B_{n, 0})$ and $\mathcal{S}(B_{n, 1})$ are generated by $B_{n, 0}$ and $B_{n, 1}$, respectively. In particular, $(B_n)_1 = B_{n, 0} \sqcup B_{n, 1}$. Furthermore, let $\mathcal{S}((Mb_{n})_1)$ be a submodule of $\mathcal{S}_{2,\infty}(\mathit{Mb} \  \hat{\times} \ I, \{x_i\}_1^{2n})$ generated by the elements of $(Mb_{n})_1$. Define a linear map $\varphi: \mathcal{S}((Mb_{n})_1) \mapsto \mathcal{S}(B_{n, 1})$ on the basis as follows:

\begin{enumerate}
    \item If $m \in Mb_{n,0}$, then there exists a unique element in $B_{n, 0}$, say $b_{n, 0}$, obtained from $m$ by replacing the crosscap with a table. In this case $\varphi(m) = b_{n,0}\in B_{n,0}$.
    \item If $m \in Mb_{n,1}$, then there exists a unique element in $B_{n, 1}$, say $b_{n,1}$, obtained from matching the arcs whose boundary are disjoint from the inner boundary. In this case $\varphi(m)=b_{n,1} \in B_{n,1}$.
\end{enumerate}
\end{definition}

The next lemma gives a direct connection between type $(Mb)_1$ and a new Gram determinant of type $B$ constructed by the set $(B_n)_1$ and using the same bilinear form as type $B$. Type $(Mb_n)_1$ is a special case of it when the distinction of the two curves attached to the inner boundary (or outer boundary) of the annulus is ignored as shown in Figure \ref{fig:labellingAnnulus}.

\begin{lemma}\label{lemma:bijection}
The map $\varphi$ is a bijection between the bases $(Mb_{n})_1$ and $(B_n)_1$. Furthermore, the Gram determinant is preserved up to an appropriate labelling of the elements.
\end{lemma}

\begin{proof}
By construction, $\varphi$ is a bijection  between $(Mb_{n})_1$ and $(B_n)_1$. To show that $\varphi$ preserves the Gram determinant it suffices to prove that the bilinear form is preserved. Since $\langle \varphi(m_i), \varphi(m_j) \rangle_{B}$ belongs to 
\begin{eqnarray*}
    \mathcal{S}_{2,\infty}(\mathit{Ann} \times I) &\oplus &\mathcal{S}_{2,\infty}(\mathit{Ann} \times I,  \{ y^{inn}_1, y_2^{inn}\}) \\
    & \oplus & \mathcal{S}_{2,\infty}(\mathit{Ann} \times I,  \{ y^{out}_1, y_2^{out}\}) \  \oplus  \ \mathcal{S}_{2,\infty}(\mathit{Ann} \times I, \{ y^{inn}_1, y_2^{inn}\} \cup \{ y^{out}_1, y_2^{out}\}),
\end{eqnarray*}
 then we only need to choose the following labelling illustrated in Figure \ref{fig:labellingAnnulus}. That is, for example, if one arc is attached to the inner and outer boundary then there exists a corresponding arc that also intersects the inner and outer boundary; we label the element with these pair of arcs by $w$. This corresponds to the element in $\mathcal{S}_{2, \infty}(\mathit{Kb} \ \hat{x} \ I)$ that intersects the two crosscaps once, namely the $w$ curve.
Furthermore, if an arc is attached to the outer boundary then we label it $y$; this corresponds to the $y$ curve in the Klein bottle that intersects the outer crosscap. If it is attached to only the inner boundary we label it $x$. In particular, the element with two arcs, one only attached to the inner boundary and one only attached to the outer boundary, is labeled $xy$. 

\begin{figure}[ht]
    \centering
    $$ \vcenter{\hbox{
\begin{overpic}[scale = .3]{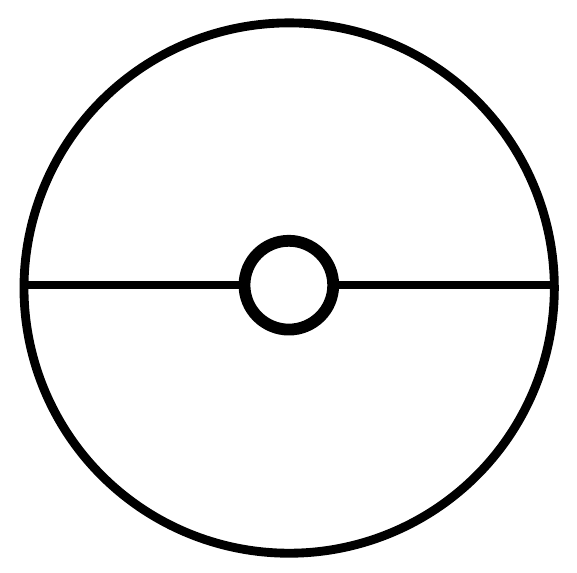}
\end{overpic} }} = w, \vcenter{\hbox{
\begin{overpic}[scale = .3]{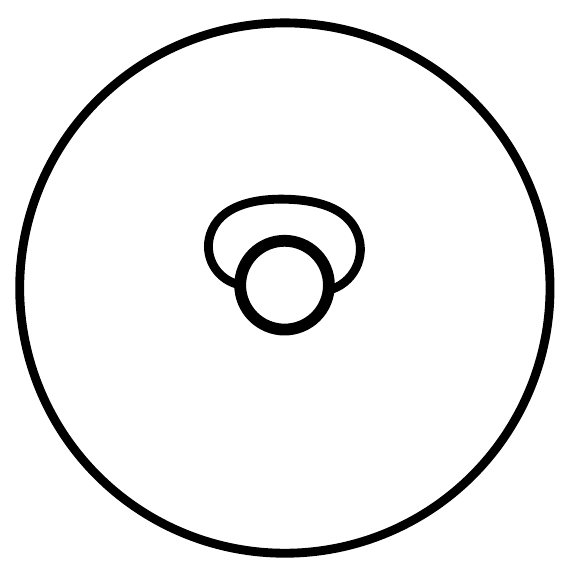}
\end{overpic} }} = x, \vcenter{\hbox{
\begin{overpic}[scale = .3]{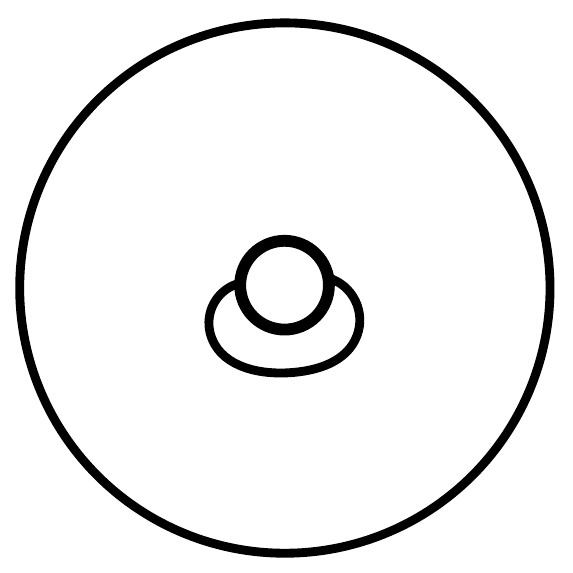}
\end{overpic} }} = x, \vcenter{\hbox{
\begin{overpic}[scale = .3]{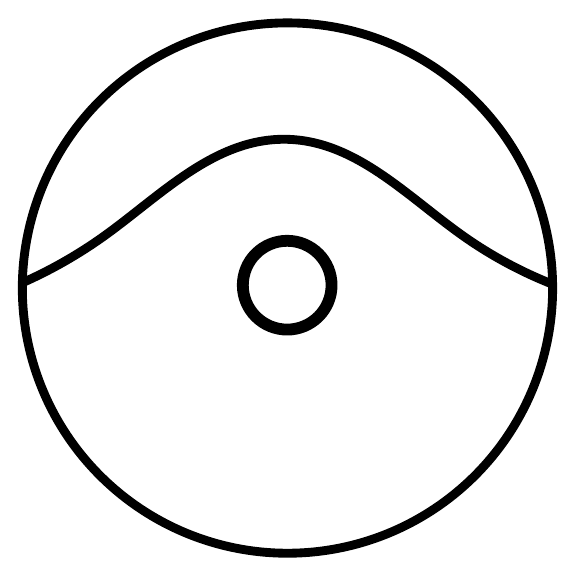}
\end{overpic} }} = y, \vcenter{\hbox{
\begin{overpic}[scale = .3]{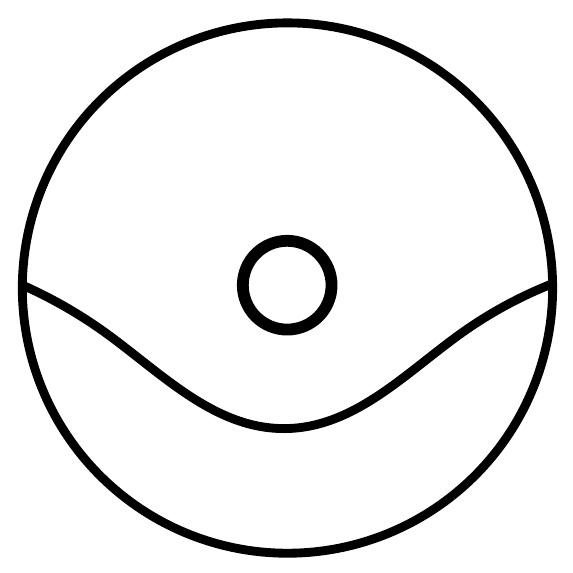}
\end{overpic} }} = y. $$
\caption{A labeling of the curves attached to the boundary of the annulus.}
\label{fig:labellingAnnulus}
\end{figure}
\end{proof}

\begin{remark}
    Even though the bijection $\varphi$ can be extended to $Mb_n$ it will no longer preserve the bilinear form. Indeed this can be seen in Example \ref{example:mbn}; when two arcs intersect the crosscap you might obtain a $d$ or $z$ for the Klein bottle while for the annulus case it would just be $x^2$.
\end{remark}

\subsection{White markers} We introduce elements whose closure remains unlinked from all simple closed curves or arcs in $\mathit{Ann} \times I$. These elements will be used to adapt Chen and Przytycki's proof in \cite{CP} to prove that a factor of the Gram determinant of type $B$ divides the Gram determinant of type $(Mb)_1$. 

\begin{definition} Let $I = [-1,1]$, $M = \mathit{Ann} \times I$ with $2n$ framed marked points attached to the outer boundary of $\mathit{Ann} \times \{ 0 \}$, and let $M$ be decorated with a lollipop $S$. Consider $\mathcal{L} \sqcup \{w_j\}_{j=1}^k \subset M$; a relative framed link $\mathcal{L}$ and a pair of $k$ labelled white marked framed points $\{w_j\}_{j=1}^k$ disjoint from $\mathcal{L}$ and distinct from the $2n$ marked points, called \textbf{white markers}. White markers are either attached and fixed to the boundary of $\mathit{Ann} \times \{ 0 \}$ of $M$, or attached to a second white marker in the interior of $M$. The labelled white markers are identified in $M$ with framed arcs according to their labelling if, up to ambient isotopy, there exists a collection of $k$ pairwise disjoint curves $\{ \gamma_j \}_{j=1}^k$ such that for each $j$, $\gamma_j$ is connected to the core of the framed points $pt_{1,j}$ and $pt_{2,j}$ of the pair of white markers $w_j = \{pt_{1,j}, pt_{2,j}\}$ under the following condition: there exists a neighborhood of each arc $\gamma_j$, $U(\gamma_j) \subset \mathit{Ann} \times \{0\}$, such that $(\cup_{i=1}^k U(\gamma_j)\times I )\cap (\mathcal{L} \cup S) = \varnothing$. 
If the condition is not satisfied then the white markers remain as white marked framed points with attaching information intact but no arcs attached. Furthermore, the sign of each labelling assigned to a white marker will change under an inversion operation.
\end{definition}

\begin{example} \ 
\begin{itemize}
\item[(a)] Let $I = [-1,1]$, $M = \mathit{Ann} \times I$ with $2n$ framed marked points attached to the outer boundary of $\mathit{Ann} \times \{ 0 \}$, and let $M$ be decorated with a lollipop $S$. Suppose that $\mathcal{L}$ consists of only relative framed links with no crossings and suppose $\mathcal{L} \cup w_1 \subset M$ where the pair of white markers $w_1 = \mathit{pt}_{1,1} \cup \mathit{pt}_{2, 1}$ are attached to each other in the interior of $M$. Then there exists a disk $D \subset \mathit{Ann} \times \{0\}$ such that $ w_1 \subset D$ and $(D \times I) \cap (\mathcal{L} \cup S)= \varnothing$. In this case the pair of labelled white markers are identified by $\gamma_1$ in $D$ to produce a simple closed curve. An illustration is given in Figure \ref{fig:whitemarkerclosure1}.

\item [(b)] Now suppose $\mathcal{L} \cup \{w_1, w_2\} \subset M$ where each white marker from $w_1$ is attached to a white marker from $w_2$ in the interior of $M$. Then there exists a disk $D \subset \mathit{Ann} \times \{0\}$ such that $ \{w_1 \cup w_2 \} \subset D$ and $(D \times I) \cap (\mathcal{L} \cup S) = \varnothing$. In this case the pair of white markers are identified by arcs in $D$. An illustration is given in Figure \ref{fig:whitemarkerclosure2}.
\end{itemize}
\end{example}

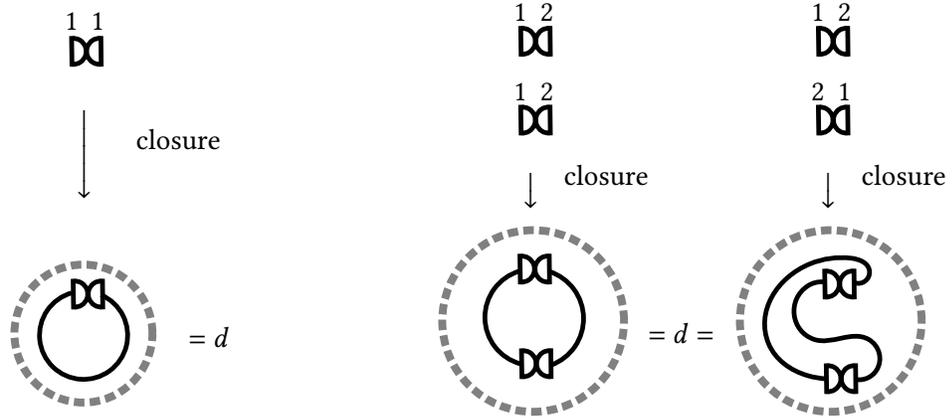
\begin{figure}[ht]
    \centering
    \begin{subfigure}{.4\textwidth}
\centering
    $$\vcenter{\hbox{\begin{overpic}[scale = .7]{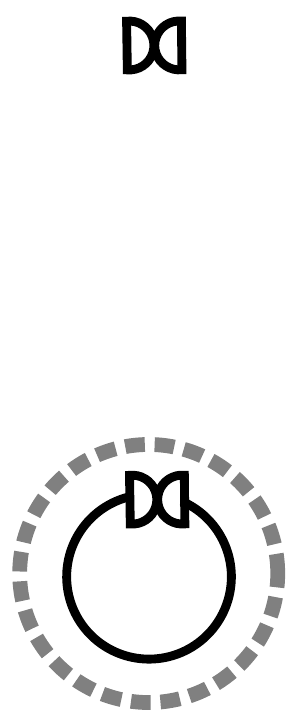}
    %\put(-20, 140){$w_1$}
    \put(33, 145){$1$}
    \put(23, 145){$1$}
    \put(50, 100){closure}
   % \put(55, 90){ $\gamma_{w_1}$}
    \put(27, 95){$\Bigg\downarrow$}
    \put(70, 25){$=d$}
   % \put(-20, 30){$\gamma_{w_1}$}
\end{overpic} }}$$
    \caption{Closure of one set of white markers.}
     \label{fig:whitemarkerclosure1}
\end{subfigure} \qquad
\begin{subfigure}{.46\textwidth}
\centering
    $$ \vcenter{\hbox{\begin{overpic}[scale = .7]{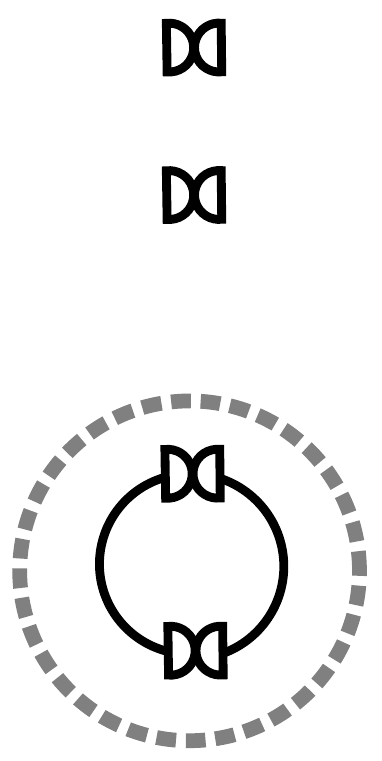}
   \put(50, 90){closure}
   \put(31, 152){$1$}
     \put(41, 152){$2$}
     \put(31, 122){$1$}
     \put(41, 122){$2$}
      \put(82, 30){$=d=$}
    \put(34, 85){$\big\downarrow$}
\end{overpic} }} \qquad \quad \vcenter{\hbox{\begin{overpic}[scale = .7]{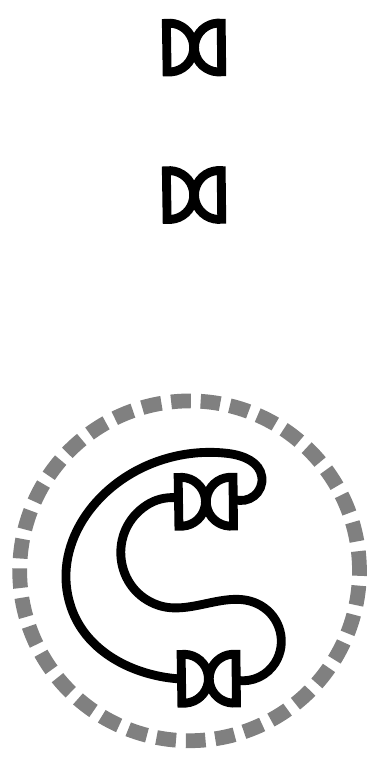}
   \put(50, 90){closure}
   \put(31, 152){$1$}
     \put(41, 152){$2$}
     \put(31, 122){$2$}
     \put(41, 122){$1$}
    \put(34, 85){$\big\downarrow$}
\end{overpic} }}$$
    \caption{Closure with two sets of white markers.}
     \label{fig:whitemarkerclosure2}
\end{subfigure}
    \caption{A local illustration of the closure of arcs in $D \times \{0\} \subset \mathit{Ann} \times [-1, 1]$ obtained by identifying arcs attached to white markers with the same label.}
    \label{fig:whitemarkerexamples}
\end{figure}

So far this notation is trivial. However, if we consider white markers attached to the boundary of $M$, as shown in Figure \ref{fig:whitemarkerexampleobstruction}, we find that there exist relative framed links that obstruct the closure of the white markers. Furthermore, we find obstructions from the position of the white markers regardless of the relative link. 

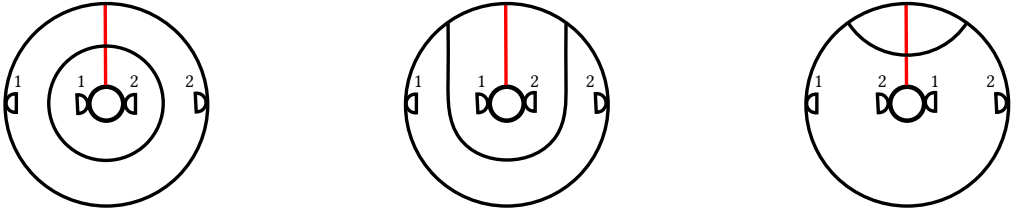
\begin{figure}[ht]
    \centering
    \begin{subfigure}{.27\textwidth}
\centering
    $$\vcenter{\hbox{\begin{overpic}[scale = .5]{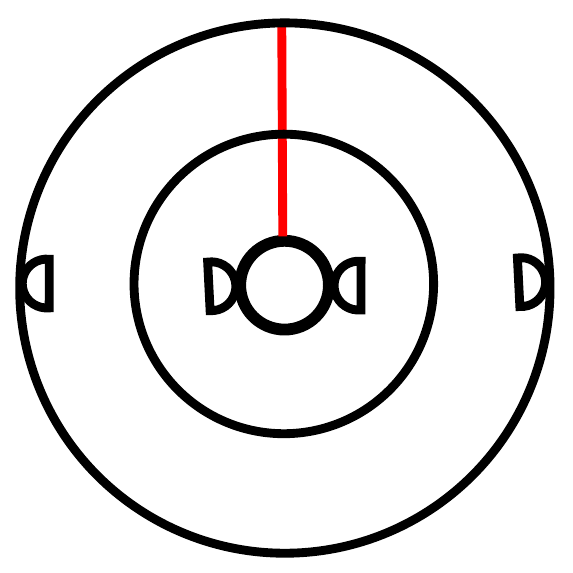}
     \put(6, 48){\fontsize{7}{7}$1$}
     \put(50, 48){\fontsize{7}{7}$2$}
     \put(30, 48){\fontsize{7}{7}$1$}
     \put(71, 48){\fontsize{7}{7}$2$}
\end{overpic} }}$$
    \caption{An obstruction from the $z$ curve isolating the inner and outer boundary points.}     \label{fig:whitemarkerexampleobstruction1}
\end{subfigure} \qquad
\begin{subfigure}{.27\textwidth}
\centering
    $$ \vcenter{\hbox{\begin{overpic}[scale = .5]{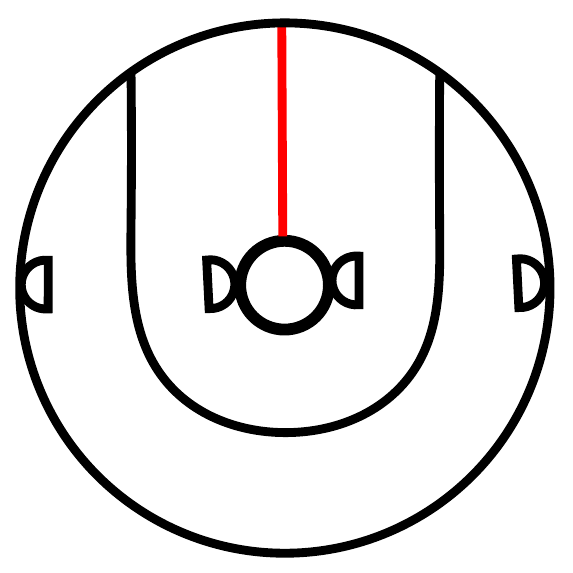}
    \put(6, 48){\fontsize{7}{7}$1$}
     \put(50, 48){\fontsize{7}{7}$2$}
     \put(30, 48){\fontsize{7}{7}$1$}
     \put(71, 48){\fontsize{7}{7}$2$}
\end{overpic} }} $$
    \caption{An obstruction coming from an arc isolating the inner and outer boundary points.} \label{fig:whitemarkerexampleobstruction2}
\end{subfigure} \qquad
\begin{subfigure}{.27\textwidth}
\centering
    $$ \vcenter{\hbox{\begin{overpic}[scale = .5]{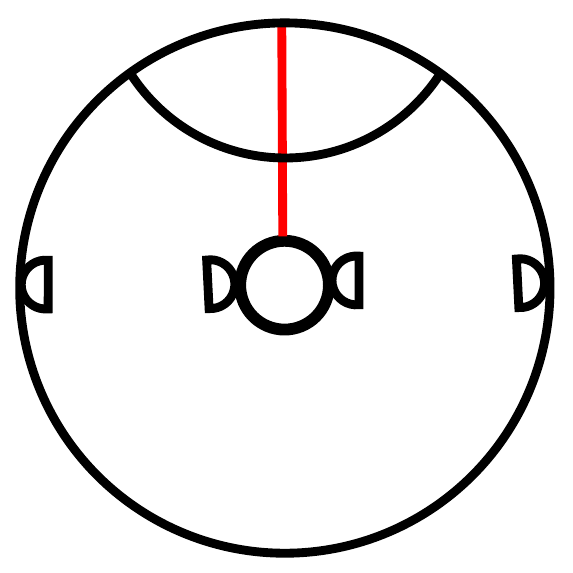}
   \put(6, 48){\fontsize{7}{7}$1$}
     \put(50, 48){\fontsize{7}{7}$1$}
     \put(30, 48){\fontsize{7}{7}$2$}
     \put(71, 48){\fontsize{7}{7}$2$}
\end{overpic} }} $$
    \caption{An obstruction coming from the position of the white markers.}
\label{fig:whitemarkerexampleobstruction3}
\end{subfigure}
    \caption{Examples of obstructions to the closure of white markers.} \label{fig:whitemarkerexampleobstruction}
\end{figure}

\begin{definition}
 Let $I = [-1,1]$, $k \leq n$, $M = \mathit{Ann} \times I$ with $2n$ framed marked points attached to the outer boundary of $\mathit{Ann} \times \{ 0 \}$, and let $M$ be decorated with a lollipop $S$. Let $\mathcal{L}^{fr}(2n, k)$ be the set of all $k$ white markers fixed to the boundary of $\mathit{Ann} \times \{ 0 \}$ along with all relative framed links $\mathcal{L}$ disjoint from the white markers,  $\mathcal{L} \sqcup \{ w_j \}_{j=1}^k$, up to ambient isotopy, while keeping all marked points on the boundary fixed. Furthermore, we restrict the $k$ white markers to finitely many possible placements in the inner boundary and also restrict to allowing at most one white marker between framed points and also between the $1^{st}$ and $2n^{th}$ framed points and the lollipop on the outer boundary.  Let $R$ be a commutative ring with unity, $A \in R$ be invertible, and let $S^{sub}_{2, \infty}(2n, k)$ be the submodule of $R \mathcal{L}^{fr}(2n, k)$ that is generated by the Kauffman bracket skein relations. The \textbf{relative Kauffman bracket skein module of $\boldsymbol{M}$ with white markers} is the quotient
 $$\mathcal{S}_{2,\infty}(\mathit{Ann} \times I, \{x_i\}_1^{2n} \cup \{ w_j \}_{j=1}^k; R, A) = R\mathcal{L}^{\mathit{fr}}(2n, k) / S_{2,\infty}^{\mathit{sub}}(2n, k).$$ 
 
 For simplicity we will denote this skein module by $\mathcal{S}(\mathit{Ann}_{k}^{n})$.
\end{definition}

\begin{corollary}
The relative Kauffman bracket skein module of $\mathit{Ann} \times I$ with white markers, $\mathcal{S}(\mathit{Ann}_{k}^{n})$, is a free $R$ module. The basis contains an infinite number of elements described as follows.

\begin{enumerate}
    \item Crossingless connections between $2n+2k$ framed points in the outer boundary of $\mathit{Ann} \times I$ and $i$ number of boundary parallel curves where $i \geq 0$.
    \item Crossingless connections between $2n+2k$ framed points in $\partial (\mathit{Ann} \times \{0\})$ and $i$ number of boundary parallel curves where $i \geq 0$, $2l$ framed points are in the inner boundary for $l \leq k$, and the arcs attached to the $2l$ framed points are not connected to the outer boundary.
    \item Crossingless connections between $2n+2k$ framed points in $\partial (\mathit{Ann} \times \{0\})$ where $l$ framed points are in the inner boundary for $l < 2k$ and at least one arc that is connected to one of the $l$ framed points is also connected to one of the $2k-l$ framed points lying in the outer boundary.
    \item Crossingless connections between $2n$ framed points in the outer boundary  of $\mathit{Ann} \times \{0\}$, $i$ boundary parallel curves where $i >0$, and $k$ white markers $\{w_j\}_{j=1}^k$ in $\partial (\mathit{Ann} \times \{0\})$, where at least one white marker, say $pt
    _{1,l}$ of $w_l=\{pt_{1,l}, pt_{2,l}\}$, lies in the inner boundary of $\mathit{Ann} \times \{0\}$ and $pt_{2,l}$ lies in the outer boundary. 
     \item Crossingless connections between $2n$ framed points on  the outer boundary $\partial (\mathit{Ann} \times \{0\})$  and $k$ white markers $\{w_j\}_{1}^k$ in $\partial (\mathit{Ann} \times \{0\})$, where for at least one white marker a path connecting it is obstructed by an arc connected to one pair of framed points or by the lollipop; see Figure \ref{fig:whitemarkerexampleobstruction}. For each element, say $b$, in this set of crossingless connections with white markers we also have $bz^i$ where $i>0$.
\end{enumerate}

\end{corollary}

\subsection{lollipop method} In this section we use the lollipop method discussed in \cite{CP} by defining similar maps. These maps are modified by using white markers so that the inner boundary is left intact and the $x, y$, and $w$ curves remain unchanged under the maps and the bilinear form.

\begin{definition} Define a linear map $\psi_k: \mathcal{S}((B_n)_1) \to \mathcal{S}((B_{n+k})_1)$ on the basis as follows. For $b \in (B_n)_1$:
    \begin{enumerate}
        \item Decorate $b$ with a lollipop $S$ and let $\pmb L$ denote the arcs whose minimum intersection points with the lollipop is equal to one.
        \item Include $k$ marked points, denoted by $\ell_k$, between the $2n^{th}$ marked point and the lollipop.
        \item Include $k$ marked points, denoted by $r_k$, between the lollipop and the $1^{st}$ marked point. 
        \item Add parallel curves denoted by $U_{b, k}$ that connects $\ell_k$ to $r_k$ in such a way that each curve crosses over $\pmb L$ on the left of the lollipop, intersects the lollipop once, and crosses under $\pmb L$ on the right of the lollipop.
    \end{enumerate}
An illustration is shown in Figure \ref{fig:psikmap}.
\end{definition}

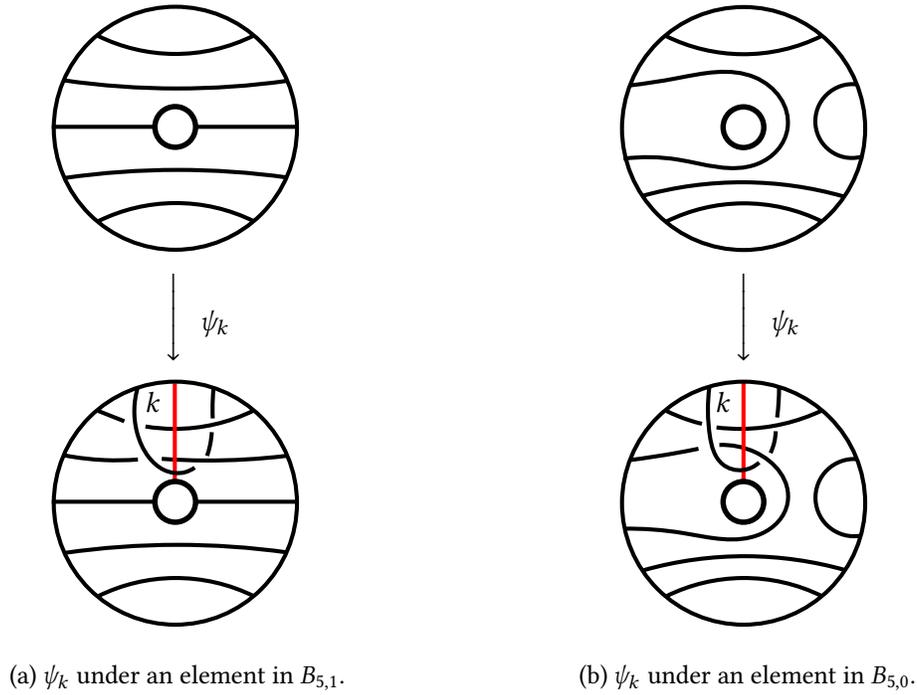
\begin{figure}[ht]
    \centering
    \begin{subfigure}{.4\textwidth}
\centering
    $$\vcenter{\hbox{\begin{overpic}[scale = .6]{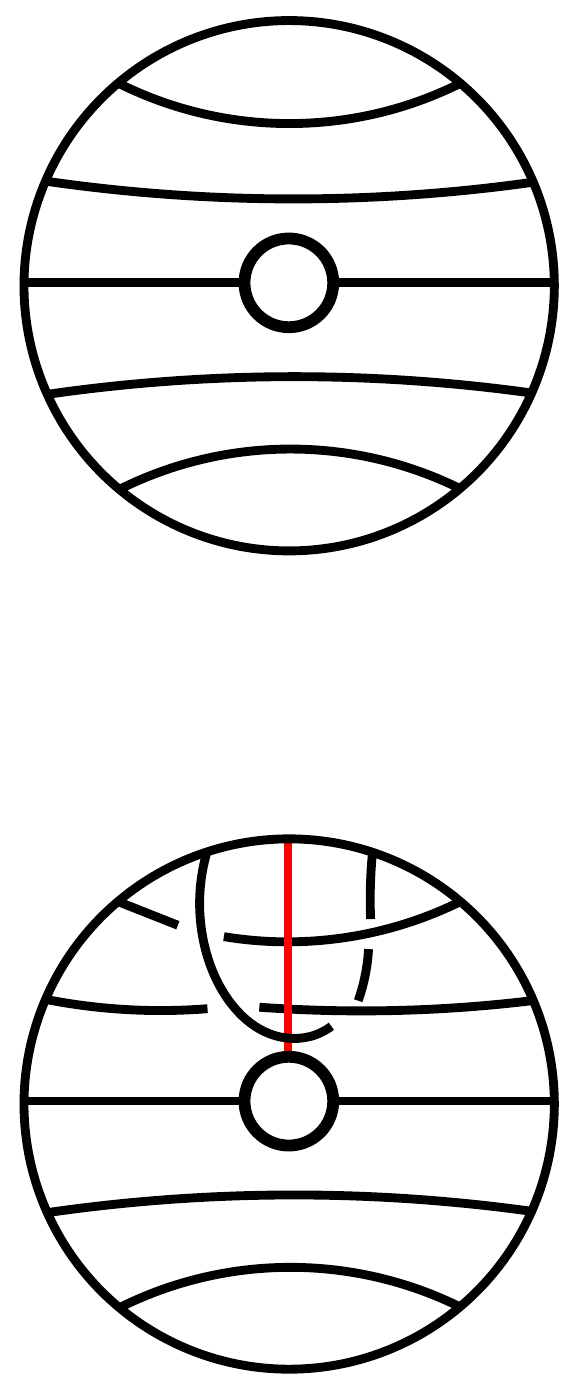}
    %\put(-20, 140){$w_1$}
    \put(60, 115){$\psi_k$}
    \put(46, 117){$\Bigg\downarrow$}
    \put(39, 84){$k$}
\end{overpic} }}$$
    \caption{$\psi_k$ under an element in $B_{5, 1}$.}
     \label{fig:psikmap1}
\end{subfigure} \quad
\begin{subfigure}{.46\textwidth}
\centering
    $$\vcenter{\hbox{\begin{overpic}[scale = .6]{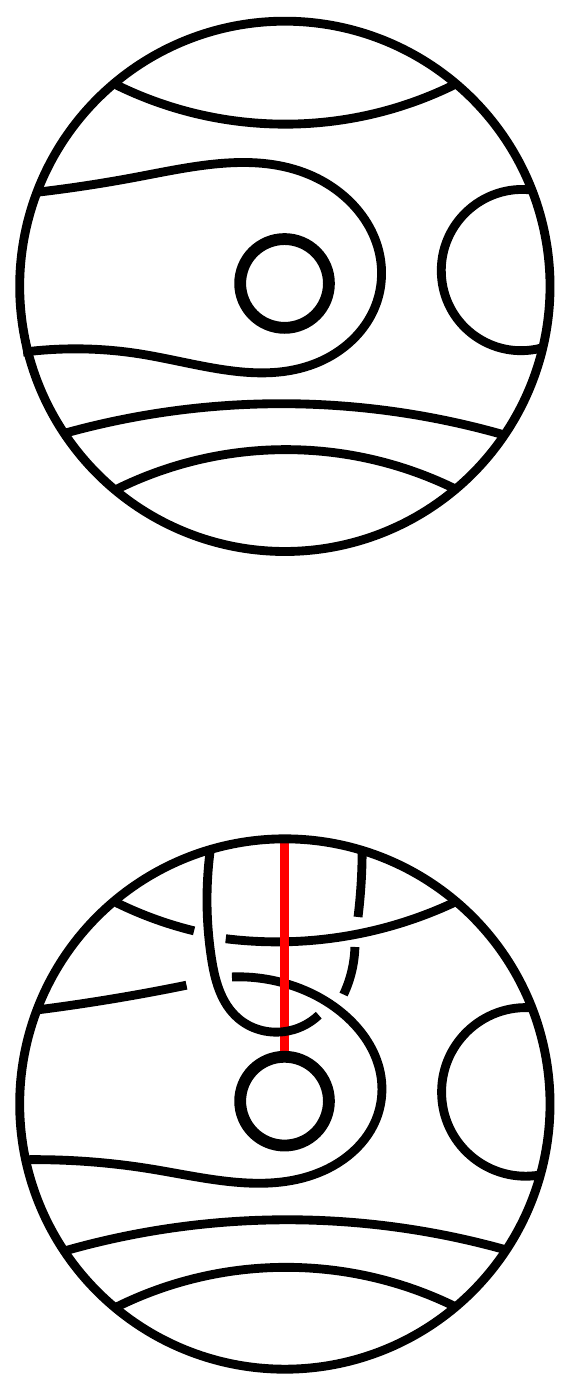}
    %\put(-20, 140){$w_1$}
    \put(60, 115){$\psi_k$}
    \put(46, 117){$\Bigg\downarrow$}
    \put(39, 84){$k$}
\end{overpic} }}$$
    \caption{$\psi_k$ under an element in $B_{5, 0}$.}
     \label{fig:psikmap2}
\end{subfigure}
    \caption{An illustration of the map $\psi_k$ under two basis elements of $\mathcal{S}((B_5)_1)$.}
    \label{fig:psikmap}
\end{figure}

\begin{definition} Define a linear map $\beta_k: \mathcal{S}((B_{n+k})_1) \mapsto \mathcal{S}(B_{n+k,0})\oplus \mathcal{S}(\mathit{Ann}_{4}^{n+k})$ on the basis as follows. For $b \in (B_{n+k})_1$:
\begin{enumerate}
    \item If $b \in B_{n+k, 0}$, then
    \begin{enumerate}
        \item decorate $b$ with a lollipop $S$ and let $\pmb L$ denote the arcs whose minimum intersection points with the lollipop is equal to one. Denote by $\ell_k$ the first $k$ marked points from the left of the lollipop and $r_k$ by the first $k$ points from the right of the lollipop.
        \item Push $\pmb L$ to the other side of the inner boundary of $\mathit{Ann} \times \{0\}$.
        \item Insert a copy of the $k^{th}$ Jones-Wenzl idempotent, $f_k$, close to $\ell_k$, into the arcs connected to $\ell_k$,  and another copy close to $r_n$ into the arcs connected to $r_k$.
    \end{enumerate}
    \item If $b \in B_{n+k, 1}$, then
    \begin{enumerate}
        \item Decorate the arc attached to $y_1$ with a labelled white marker labelled $1$.
        \item Decorate the arc attached to $y_2$ with a labelled white marker labelled $2$.
        \item Push $\pmb L \cup U_b$ to the other side of the inner boundary of $\mathit{Ann} \times \{0\}$.
        \item Insert two copies of the $k^{th}$ Jones-Wenzl idempotent, $f_k$, into $U_{b, k}$, one close too $\ell_n$ and another close to $r_n$.
    \end{enumerate}
\end{enumerate}
An illustration is shown in Figure \ref{fig:betakmap}.
\end{definition}

\begin{figure}[ht]
    \centering
    \begin{subfigure}{.4\textwidth}
\centering    $$\vcenter{\hbox{\begin{overpic}[scale = .6]{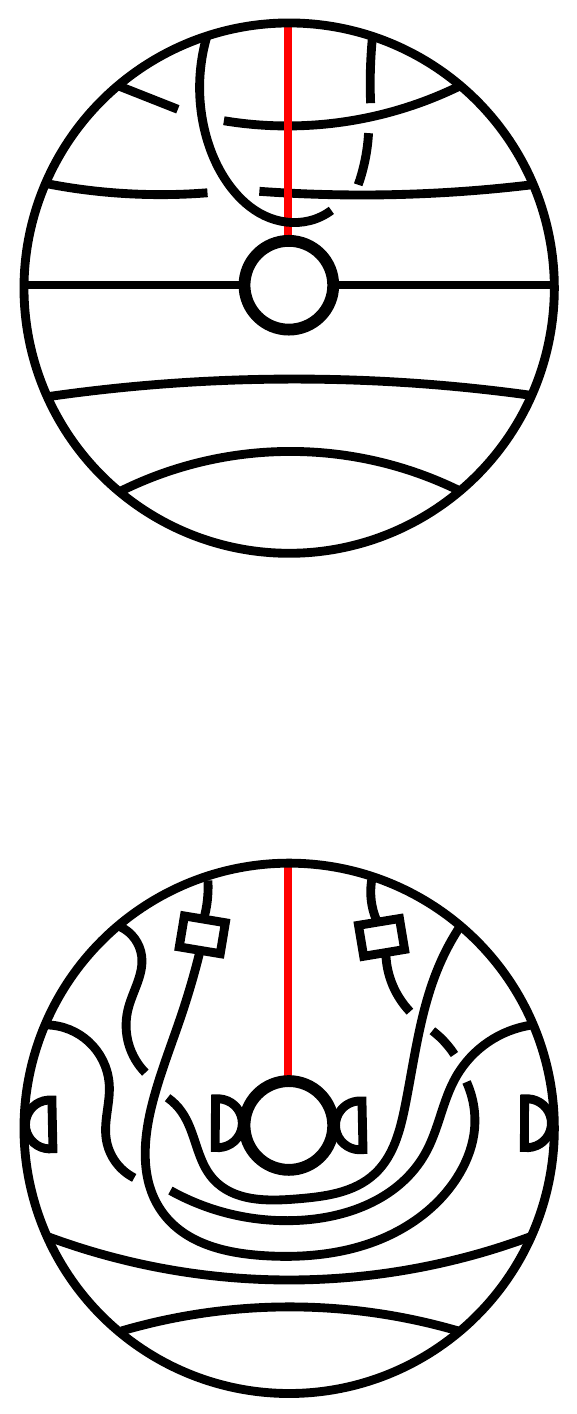}
    %\put(-20, 140){$w_1$}
    \put(60, 115){$\beta_k$}
     \put(37, 229){$k$}
    \put(46, 117){$\Bigg\downarrow$}
    \put(39, 87){\fontsize{8}{8}$k$}
     \put(57, 87){\fontsize{8}{8}$k$}
     \put(36, 57){\fontsize{8}{8}$1$}
     \put(7, 57){\fontsize{8}{8}$1$}
      \put(60, 57){\fontsize{8}{8}$2$}
     \put(89, 57){\fontsize{8}{8}$2$}
\end{overpic} }}$$
    \caption{$\beta_k$ under an element in $B_{5+k, 1}$.}
     \label{fig:betakmap1}
\end{subfigure} \quad
\begin{subfigure}{.46\textwidth}
\centering    $$\vcenter{\hbox{\begin{overpic}[scale = .6]{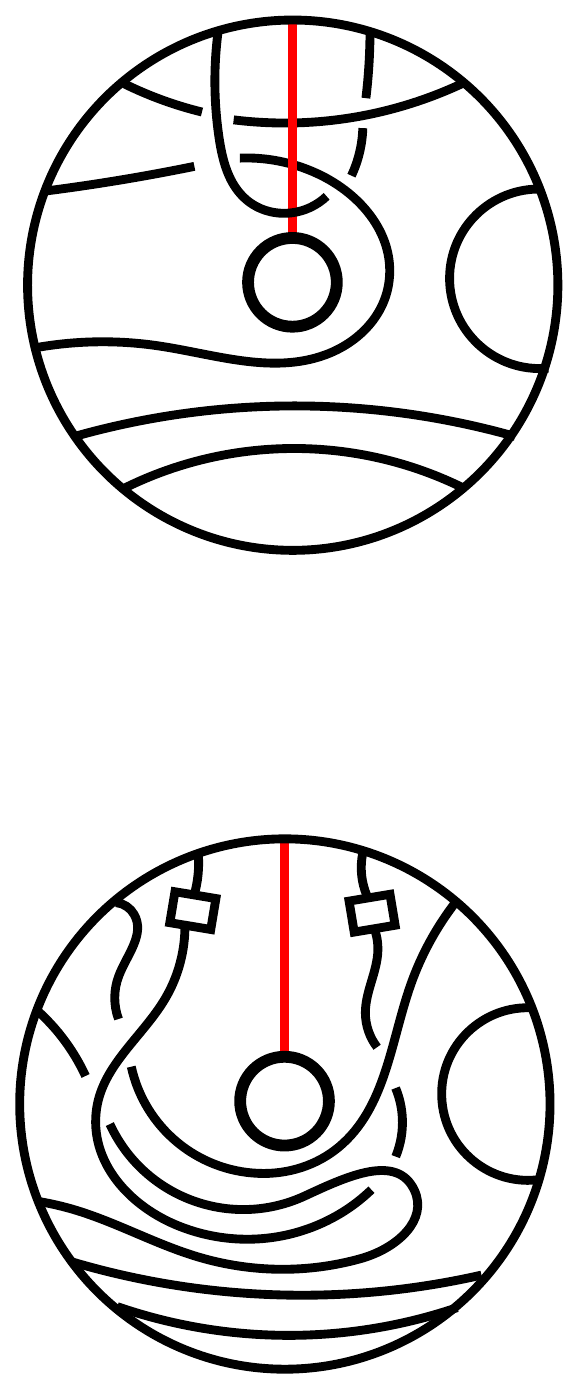}
    %\put(-20, 140){$w_1$}
    \put(60, 115){$\beta_k$}
     \put(39, 226){$k$}
    \put(46, 117){$\Bigg\downarrow$}
    \put(38, 87){\fontsize{8}{8}$k$}
     \put(56, 87){\fontsize{8}{8}$k$}
\end{overpic} }}$$
    \caption{$\beta_k$ under an element in $B_{5+k, 0}$.}
     \label{fig:betakmap2}
\end{subfigure}
    \caption{An illustration of the map $\beta_k$ under two elements of $(B_{5+k})_1$.}
    \label{fig:betakmap}
\end{figure}
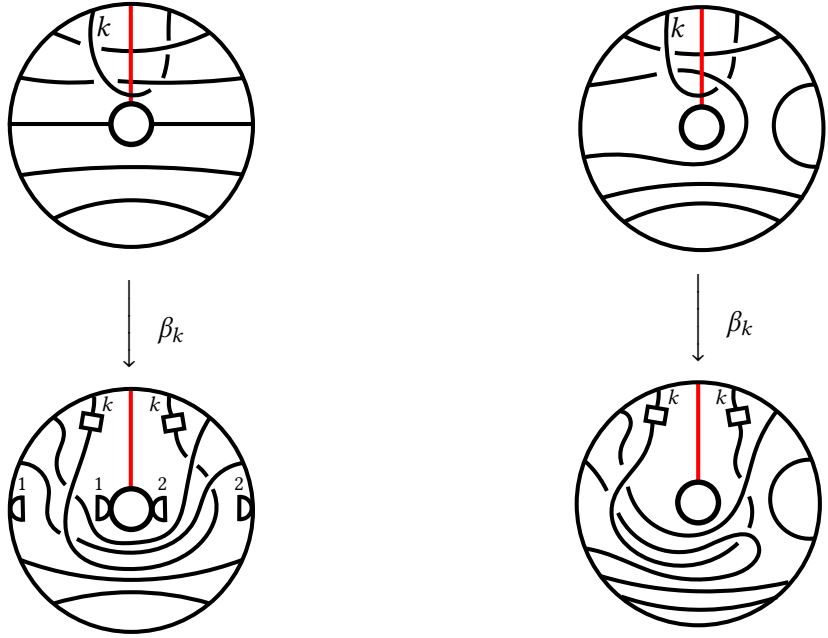

The next lemma is a direct result of the previous definition. 

\begin{lemma} Let $b_i, b_j \in (B_n)_1$.
If $\langle b_i, b_j \rangle_B = x^py^mw^hd^nz^r$ for $p, m, h, n, r \geq 0$, then 
$$\langle \beta_0 \circ \psi_0(b_i), \beta_0 \circ \psi_0(b_j) \rangle_B = x^py^mw^hd^{n+r}.$$ 
That is, all $z$ curves become homotopically trivial. 
\end{lemma}

\begin{definition}
Let $H(x, y) \in \mathcal{S}_{2, \infty}(S^3)$ denote the Hopf link in $S^3$ decorated by $x$ and $y$. We define a linear map $\xi_k: \mathcal{S}_{2, \infty}(\mathit{Kb} \ \hat{\times}\ I) \to \mathbb{Z}[A^{\pm 1}x,y,w]$ such that $\xi_k$ is the identity on $x, y, w,$ and $ d$, and
$$ \xi_k (z) := H(z, f_k),$$
where the second component is decorated with the $k^{th}$ Jones-Wenzl idempotent $f_k$.
\end{definition}

Recall from W. B. R. Lickorish in \cite{Lic} that 
\begin{equation}
\vcenter{\hbox{
\begin{overpic}[scale=.3]{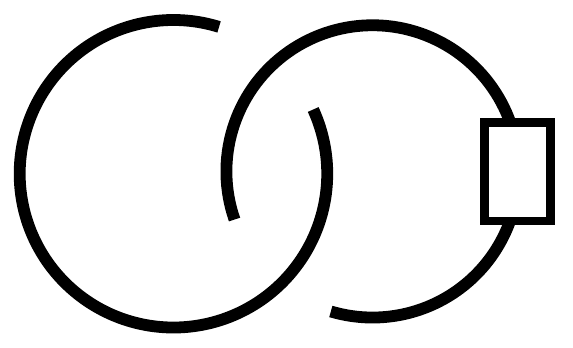}
\put(-5,0){\fontsize{9}{9}$m$}
\put(45,0){\fontsize{9}{9}$k$}
\end{overpic}}} = (-A^{2(k+1)}-A^{-2(k+1)})^m \Delta_k = ((-1)^{k}T_{k+1})^m(d) \Delta_k.
\end{equation}
Therefore, 
$$ \xi_k(z^m) = H(z^m, f_k) = \vcenter{\hbox{
\begin{overpic}[scale=.3]{Hopffn1.pdf}
\put(-5,0){\fontsize{9}{9}$m$}
\put(45,0){\fontsize{9}{9}$k$}
\end{overpic}}} =  (-A^{2(k+1)}-A^{-2(k+1)})^m \Delta_k.$$

\begin{example} Let $d = -A^2-A^{-2}$. Then
$$\xi_k \left( \vcenter{\hbox{
\begin{overpic}[scale=.3]{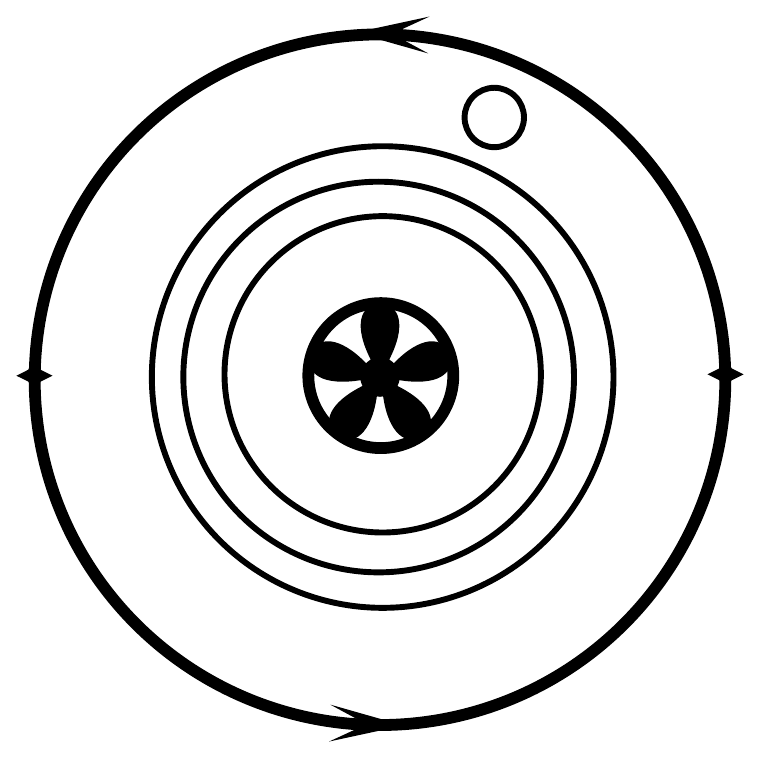}
\end{overpic}}} \right) = dH(z^3, f_k) = d(-A^{2(k+1)}-A^{-2(k+1)})^3 \Delta_k = d((-1)^{k}T_{k+1}(d))^3\Delta_k,$$

$$\xi_k \left( \vcenter{\hbox{
\begin{overpic}[scale=.3]{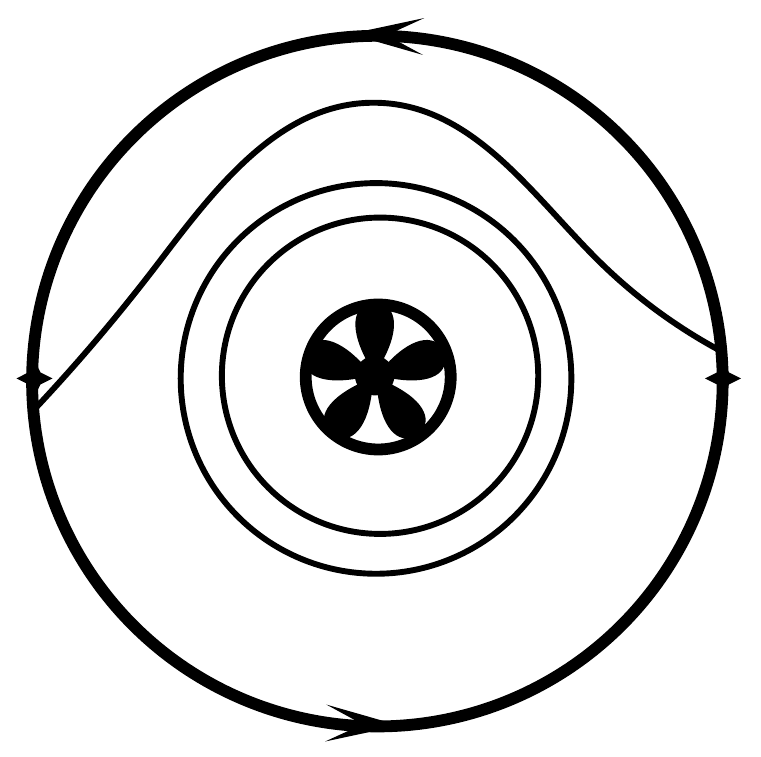}
\end{overpic}}} \right) = yH(z^2, f_k)= y((-1)^{k}T_{k+1}(d))^2\Delta_k,$$

and 
$$\xi_k \left( \vcenter{\hbox{
\begin{overpic}[scale=.3]{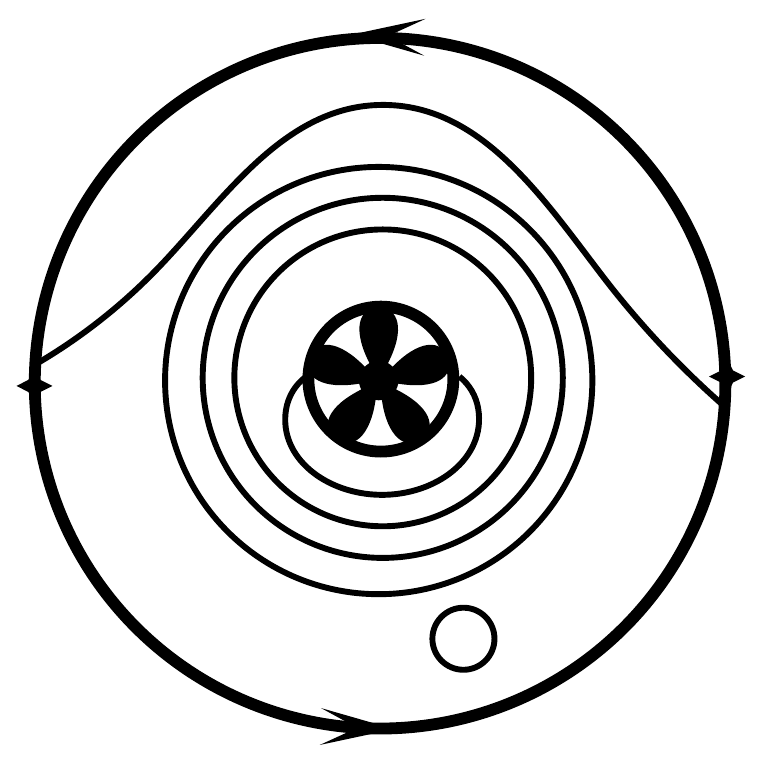}
\end{overpic}}} \right) = dxy((-1)^{k}T_{k+1}(d))^3\Delta_k.$$
\end{example}

\begin{definition}
Define the matrix $F_{n, k}$ using $\xi_k$ as follows. For $m_i, m_j \in (\mathit{Mb}_n)_1$,
\begin{equation*}
    F_{n, k} = (\xi_{k-1}(\langle m_i, m_j\rangle_{\mathit{Mb}}))_{1 \leq i, j \leq \binom{2n}{n-1}+\binom{2n}{n}}.
\end{equation*}
\end{definition}
 
\begin{remark} By the definition, if $d=-A^2-A^{-2}$, then
$$G^{(\mathit{Mb})_1}_n(d,(-1)^{k-1} T_k(d), x, y, w) = \frac{1}{\Delta_{k-1}} F_{n,k}.$$

Furthermore, if we define $F'_{n,k}$ on $\mathit{Mb}_n$ in a similar way, $ F'_{n, k} = (\xi_{k-1}(\langle m'_i, m'_j\rangle_{\mathit{Mb}}))_{1 \leq i, j \leq |\mathit{Mb}_n|},$
then 
$$G^{\mathit{Mb}}_n(d,(-1)^{k-1} T_k(d), x, y, w) = \frac{1}{\Delta_{k-1}} F'_{n,k}.$$
\end{remark}

The following lemma is a direct result of the construction of the maps $\beta_k$, $\psi_k$, and $\varphi$. Notice that $\varphi$ is a bijection between the sets $(Mb_n)_1$ and $(B_n)_1$ which preserves the bilinear form, $\psi_k$ introduces $k$ trivial link components linked to the $z$ curves under the bilinear form of type $B$, and $\beta_k \circ \psi_k$, under the bilinear form of type $B$, inserts the $k^{th}$ Jones-Wenzl idempotent into the $k$ trivial link components. Additionally, $\beta_k \circ \psi_k$ pushes all boundary parallel curves through the inner boundary of the annulus while preserving the number of $x, y$, and $w$ curves.
\begin{lemma}\label{Lemma:maps} Let $\eta_k = \beta_k \circ \psi_k \circ \varphi$. For $m_i, m_j \in (\mathit{Mb}_n)_1$, 
$$\xi_{k}(\langle m_i, m_j\rangle_{\mathit{Mb}}) = \langle     \eta_k (m_i) , \eta_k (m_j) \rangle_B. $$
\end{lemma}

The next theorem is our main result regarding the structure of the closed formula for the Gram determinant of type $(Mb)_1$.

\begin{theorem} \label{MainTheorem} 
$$\prod\limits_{k=1}^n(T_k(d)+(-1)^k z)^{\binom{2n}{n-k}} \mbox{ divides } D^{(\mathit{Mb})_1}_n(d, z, x, y, w).$$
\end{theorem}

\begin{proof}
We modify and build upon Chen and Przytycki's proof in \cite{CP} by showing that for $k\geq 1$, the nullity of $G^{(\mathit{Mb})_1}_n(d,(-1)^{k-1} T_k(d), x, y, w)$ for $d=-A^2-A^{-2}$ is at least $\binom{2n}{n-k}$. 

\ 

Recall that for $F_{n, k} = (\xi_{k-1}(\langle m_i, m_j\rangle_{\mathit{Mb}}))_{1 \leq i, j \leq \binom{2n}{n-1}+\binom{2n}{n}}$ we have $$G^{(\mathit{Mb})_1}_n(d,(-1)^{k-1} T_k(d), x, y, w) = \frac{1}{\Delta_{k-1}} F_{n,k}.$$

Furthermore, by Lemma \ref{Lemma:maps}, we have $\xi_{k-1}(\langle m_i, m_j\rangle_{\mathit{Mb}}) = \langle  \eta_{k-1} (m_i) , \eta_{k-1} (m_j) \rangle_B$ where $\eta_k = \beta_k \circ \psi_k \circ \varphi$. Since $\varphi$ is a bijection between $(Mb_n)_1$ and $(B_n)_1$, then it suffices to show that $\beta_{k-1} \circ \psi_{k-1}((B_n)_1)$ is contained in a subspace of dimension $\binom{2n}{n-1}+\binom{2n}{n}-\binom{2n}{n-k}$ in $\mathcal{S}_{2, \infty}(\mathit{Ann}_{4}^{n+k-1}) \oplus \mathcal{S}_{2, \infty}(D^2 \times I, n+k-1)$. As in \cite{CP}, it suffices to show that
$$\text{dim}(\text{Im}(\beta_{k-1})) \leq  \binom{2n}{n-1}+\binom{2n}{n}-\binom{2n}{n-k}.$$

Recall that $\mathcal{S}((B_n)_1) = \mathcal{S}(B_{n, 0}) \oplus \mathcal{S}(B_{n, 1})$ where $\mathcal{S}(B_{n, 0})$ and $\mathcal{S}(B_{n, 1})$ are free $R$-modules generated by $B_{n, 0}$ and $B_{n, 1}$, respectively. Since $B_{n,0} \cap B_{n, 1} = \varnothing$, then 
$$\beta_{k-1} \circ \psi_{k-1}(B_{n, 0} \oplus B_{n, 1}) = \beta_{k-1} \circ \psi_{k-1} (B_{n,0} \oplus \{ 0 \}) \oplus \beta_{k-1} \circ \psi_{k-1}( \{0 \} \oplus B_{n, 1}). $$

Furthermore, $|B_{n,1}|= \binom{2n}{n-1}$. Therefore,
$$\dim (\beta_{k-1} \circ \psi_{k-1}((B_n)_1)) \leq \dim (\beta_{k-1} \circ \psi_{k-1}(B_{n,0} \oplus \{0\}) )+ \binom{2n}{n-1}. $$

The set $\beta_{k-1} \circ \psi_{k-1}(B_{n,0} \oplus \{0\})$ can be viewed as a set of crossingless connections between $2(n+k-1)$ points in the disc by cutting along the lollipop. Therefore, it suffices to prove that $\beta_{k-1} \circ \psi_{k-1}(B_{n,0} \oplus \{0\})$ is contained in a subspace of dimension $\binom{2n}{n} - \binom{2n}{n-k}$ in $\mathcal{S}_{2, \infty}(D \times I, n+k-1)$ which was already proven in \cite{CP}. 
\end{proof}

\section{Future directions}\label{FutureDirections}

We present a conjectured formula for type $(\mathit{Mb})_1$.
%adapted by Chen's conjectured formula for type $Mb$.
This conjecture has been verified for $n \leq 3$. 
\begin{conjecture}\label{Conjecture}\

Let $R = \mathbb{Z}[A^{\pm 1},w,x,y,z].$ Then,
the Gram determinant of type $(Mb)_1$ for $n \geq 1$, is:
\begin{eqnarray*}
D^{(Mb)_1}_n &=&  \left[(d-z)((d + z)w -2xy) \right]^{\binom{2n}{n-1}} \prod_{k=2}^n (T_k(d)^2-z^2)^{\binom{2n}{n-k} } \prod\limits_{k=2}^n (T_{2k}(d)-2)^{\binom{2n}{n-k}},
\end{eqnarray*}
% \begin{eqnarray*}
%D^{(Mb)_1}_n(d,z,x,y,w) &=& \prod_{k=1}^n (T_k(d)+(-1)^kz)^{\binom{2n}{n-k} } \\
%& & \prod\limits_{ \substack{k=1 \\ k\text{ odd }}}^n
%_{k=1 \atop k \text{ odd }}^n 
%\left((T_k(d) - (-1)^k z)T_k(w) -2xy\right)^{\binom{2n}{n-k}} \\
%& & \prod\limits_{ \substack{k=1 \\ k\text{ even }}}^n
%_{ k=1 \atop k \text{ even }}^n 
%\left((T_k(d) - (-1)^kz)\right)^{\binom{2n}{n-k}} \\
%& & \prod\limits_{k=2}^n (T_{2k}(d)-2)^{\binom{2n}{n-k}},
%\end{eqnarray*}
where $T_k(d)$ is the $k^{th}$ Chebyshev polynomial of the first kind and $d = -A^2-A^{-2}$.
\end{conjecture}

The lollipop method was proposed by Przytycki to be a suitable method to apply to type $Mb$ for a proof of the same factor given in our main theorem. In fact, the concluding arguments of the proof of Theorem \ref{MainTheorem} suggest that the new techniques given in this paper along with the lollipop method can be modified and applied to type $Mb$; albeit in a not so straightforward way.

\ 

Lemma \ref{lemma:bijection} gives insight into direct connections between type $(Mb)_1$ and a Gram determinant from a matrix created by using the set $(B_n)_1$ and the bilinear form of type $B$. We will call this determinant type $(B)_1$. As shown in the proof of the lemma, type $(Mb)_1$ is a special case of type $(B)_1$ where no distinction is made between arcs whose boundary components are both connected to the inner boundary (or outer boundary).
Another future direction is to focus on the Gram determinant of type $(B)_1$ and obtain a closed formula. Furthermore, potential connections to statistical mechanics are suggested, by investigating the blob algebra discussed in \cite{MaSa2}, when the $y$ and $x$ variables are made equal as to obtain a determinant from a symmetric matrix.


\newpage
\begin{thebibliography}{9999999}
	
 \bibitem[BIMP1]{BIMP} R. P. Bakshi, D. Ibarra, S. Mukherjee, J. H. Przytycki, A generalization of the Gram determinant of type $A$, \textit{Topology Appl.} 295 (2021), Paper No. 107663, 15 pp.
e-print: \href{https://arxiv.org/abs/1905.07834}{arXiv:1905.07834} [math.GT].
\bibitem[BIMP2]{BIMP2} R. P. Bakshi, D. Ibarra, S. Mukherjee, J. H. Przytycki, A Note on the Gram Determinant of Type Mb, (to appear in AMS Contemporary Mathematics series.)

	\bibitem[Cai]{Cai} X. Cai, A Gram determinant of Lickorish's bilinear form, \textit{Math. Proc. Cambridge Philos. Soc.} 151 (2011), no. 1, 83–94. e-print: \href{https://arxiv.org/pdf/1006.1297.pdf}{arXiv:1006.1297v3} [math.GT].
	
 \bibitem[Che]{Che} Q. Chen, Personal communication (email) with J. H. Przytycki, April 3, 2009.
 \bibitem[CP1]{Ch-P2}
Q.~Chen, J.~H.~Przytycki, The Gram matrix of the Temperley-Lieb algebra is similar to the matrix of
chromatic joins,
{\it Communications in Contemporary Mathematics (CCM)},  10, 2008, 849-855; e-print:\ \href{http://front.math.ucdavis.edu/0806.0878}{http://front.math.ucdavis.edu/0806.0878}
 \bibitem[CP2]{CP} Q. Chen, J. H. Przytycki, The Gram determinant of the type $B$ Temperley-Lieb algebra, {\it Adv. in Appl. Math.}, 43(2), 2009, 156-161. \href{https://arxiv.org/abs/1006.1297}{arXiv:0802.1083v2} [math.GT]. 

\bibitem[DiF]{DiF} 
P. Di Francesco,
Meander determinants. 
{\it Comm. Math. Phys.} 191 (1998), no. 3, 543–583.
\href{https://arxiv.org/abs/hep-th/9612026}
{arXiv:hep-th/9612026} [math.GT].
 
 \bibitem[Iba]{Iba} D. Ibarra, Framed Links in 3-Manifolds, Its Applications, and Algebraic Approaches to Knot Theory. Thesis (Ph.D.)–The George Washington University. 2022. 163 pp. ISBN: 979-8802-71109-5, ProQuest LLC.

 \bibitem[KS]{K-S}
K. H. Ko, L. A. Smolinsky,  Combinatorial matrix in 3-manifold theory. {\it Pacific J. Math.} 149 (1991), no. 2, 319–336.
 
\bibitem[Lic1]{Lic1} W.~B.~R.~Lickorish, Invariants for 3-manifolds from the combinatorics of the Jones polynomial, {\it Pacific Journ. Math.},149(2), 1991, 337-347. 
 
 \bibitem[Lic2]{Lic}
W. B. R. Lickorish, An introduction to knot theory.
{\it Graduate Texts in Mathematics}, 175. Springer-Verlag, New York, 1997. 

\bibitem[MS1]{MaSa1} 
 P.~Martin, H.~Saleur, On an algebraic approach to higher dimensional statistical mechanics. {\it Commun. Math. Phys.}, 158, 1993, 155-190.
 
 \bibitem[MS2]{MaSa2}  P. Martin, H. Saleur, The blob algebra and the periodic Temperley-Lieb algebra. \textit{Lett. Math. Phys.} 30 (1994), no. 3, 189–206. \href{https://arxiv.org/pdf/hep-th/9302094.pdf}
{arXiv:hep-th/9302094}.

\bibitem [Prz1]{Prz1} J. H. Przytycki, Fundamentals of Kauffman bracket skein modules. {\it Kobe Math. J.}, 16(1), 1999, 45-66. \href{https://arxiv.org/abs/math/9809113}{arXiv:math/9809113} [math.GT].

\bibitem[Prz2]{Prz2} J. H. Przytycki, Personal communication (email) with Q. Chen, November 21, 2008.
\bibitem[PBIMW]{PBIMW} J. H. Przytycki, R. P. Bakshi, D. Ibarra, G. Montoya-Vega, D. E. Weeks, Lectures on Knot Theory: An Exploration of Contemporary Topics, \textit{Springer Universitext} (to appear).

\bibitem[RT]{RT} N. Yu. Reshetikhin and V. G Turaev, Invariants of 3-manifolds via link polynomials and quantum groups, {\it inv. math}, 103(1991), 547-597.

\bibitem[Sch]{SchTypeA} F. Schmidt, Problems related to type-A and type-B matrices of chromatic joins. Special issue on the Tutte polynomial, {\it Adv. in Appl. Math.} 32 (2004), no. 1-2, 380-390.

\bibitem[Sim]{Sim} R. Simion, A type-B associahedron. Formal power series and algebraic combinatorics (Scottsdale, AZ, 2001). {\it Adv. in Appl. Math.} 30 (2003), no. 1-2, 2-25.

\bibitem[Wit]{Wit} E. Witten,  Quantum field theory and the Jones polynomial, {\it Comm. Math. Phys.} 121 (1989), 351.
	\end{thebibliography}
\end{document}